\documentclass[10pt]{amsart}

\usepackage{amssymb,esint,color}
\usepackage[colorlinks=true,urlcolor=blue, citecolor=red,linkcolor=blue,
linktocpage,pdfpagelabels, bookmarksnumbered,bookmarksopen]{hyperref}
\usepackage[hyperpageref]{backref}
\usepackage{color}
\usepackage{verbatim}

\usepackage{geometry}

\newtheorem{lm}{Lemma}[section]
\newtheorem{teo}[lm]{Theorem}
\newtheorem{prop}[lm]{Proposition}
\newtheorem{coro}[lm]{Corollary}

\theoremstyle{definition}
\newtheorem{oss}[lm]{Remark}
\newtheorem{defi}[lm]{Definition}
\newtheorem*{ack}{Acknowledgments}

\author[Brasco]{Lorenzo Brasco}
\author[Cinti]{Eleonora Cinti}
\author[Vita]{Stefano Vita}

\address[L.\ Brasco]{Dipartimento di Matematica e Informatica
	\newline\indent
	Universit\`a degli Studi di Ferrara
	\newline\indent
	Via Machiavelli 35, 44121 Ferrara, Italy}
\email{lorenzo.brasco@unife.it}

\address[E.\ Cinti]{Dipartimento di Matematica
	\newline\indent
	Universit\`a degli Studi di Bologna
	\newline\indent
	 Piazza di Porta San Donato 5, 40126 Bologna, Italy}
\email{eleonora.cinti5@unibo.it}

\address[S.\ Vita]{Dipartimento di Matematica e Applicazioni
\newline\indent
Universit\`a degli Studi di Milano Bicocca
\newline\indent
Via Cozzi 55, 20125 Milano, Italy}

\email{stefano.vita@unimib.it}

\title[Quantitative fractional Faber-Krahn]{A quantitative stability estimate\\ for the fractional Faber-Krahn inequality}

\subjclass[2010]{47A75, 49Q20, 35R11}
\keywords{Stability of eigenvalues, fractional Laplacian}

\numberwithin{equation}{section}

\dedicatory{Remembering Rosalind Elsie Franklin on the centenary of her birth}

\begin{document}

\begin{abstract}
We prove a quantitative version of the Faber-Krahn inequality for the first eigenvalue of the fractional Dirichlet-Laplacian of order $s$.
This is done by using the so-called Caffarelli-Silvestre extension and adapting to the nonlocal setting a trick by Hansen and Nadirashvili. The relevant stability estimate comes with an explicit constant, which is stable as the fractional order of differentiability goes to $1$. 
\end{abstract}

\maketitle

\begin{center}
\begin{minipage}{11cm}
\small
\tableofcontents
\end{minipage}
\end{center}

\section{Introduction}

\subsection{The Faber-Krahn inequality}
The celebrated {\it Faber-Krahn inequality} asserts that for every open set $\Omega\subset\mathbb{R}^N$ with finite $N-$dimensional Lebesgue measure, we have the sharp estimate
\begin{equation}
\label{faber}
|\Omega|^\frac{2}{N}\,\lambda(\Omega)\ge |B|^\frac{2}{N}\,\lambda(B),
\end{equation}
where $B$ is any $N-$dimensional ball. Moreover, equality in \eqref{faber} is uniquely attained by balls.
The quantity $\lambda(\Omega)$ is the {\it first eigenvalue of the Dirichlet-Laplacian on} $\Omega$. In other words, it is the smallest real number $\lambda$ such that the boundary value problem
\[
\left\{\begin{array}{rccl}
-\Delta u&=&\lambda\,u,& \mbox{ in }\Omega,\\
u&=& 0, & \mbox{ on } \partial\Omega,
\end{array}
\right.
\]
admits a nontrivial solution $u\in \mathcal{D}^{1,2}_0(\Omega)$. The latter is the homogeneous Sobolev space, defined as the completion of $C^\infty_0(\Omega)$ with respect to the norm
\[
\varphi\mapsto \left(\int_\Omega |\nabla \varphi|^2\,dx\right)^\frac{1}{2}.
\]
By observing that $\lambda$ scales like a length to the power $-2$, it is easily seen that the inequality in \eqref{faber} is scale invariant. Moreover,
the Faber-Krahn inequality can be equivalently rephrased by saying that balls (uniquely) solve the shape optimization problem
\[
\min\{\lambda(\Omega)\, :\, \Omega\subset \mathbb{R}^N \mbox{ open set with }|\Omega|=c\},
\]
for every $c>0$. 
\par
We briefly recall that a way to prove \eqref{faber} is by using the {\it Schwarz symmetrization}. In other words, given $u$ a non-negative function, we can construct the unique radially symmetric decreasing function $u^*$ such that 
\[
|\{x\, :\, u(x)>t\}|=|\{x\, :\, u^*(x)>t\}|,\qquad \mbox{ for every } t\ge0.
\] 
By construction, the two functions $u$ and $u^*$ are equi-measurable, thus all the $L^q$ norms of $u$ and $u^*$ coincide.
Moreover, by the well-known {\it P\'olya-Szeg\H{o} principle} we know that
\begin{equation}
\label{PSintro}
\int_\Omega |\nabla u|^2\,dx\ge \int_{B_\Omega} |\nabla u^*|^2\,dx,
\end{equation}
where $B_\Omega$ is the ball centered at the origin, such that $|\Omega|=|B_\Omega|$. By using these two facts and the variational characterization
\[
\lambda(\Omega)=\min_{u\in \mathcal{D}^{1,2}_0(\Omega)} \left\{\int_\Omega |\nabla u|^2\,dx\, :\, \|u\|_{L^2(\Omega)}=1\right\},
\]
one immediately gets \eqref{faber}.
\par
Starting with the works of Hansen \& Nadirashvili \cite{HN} and Melas \cite{Me}, there has been a surge of interest towards the {\it stability issue} for the Faber-Krahn inequality. In other words, one seeks for quantitative enhancements of \eqref{faber}, containing remainder terms measuring the deviation of a set $\Omega$ from spherical symmetry. We refer to the book chapter \cite{BD} for a comprehensive treatment of the subject. Here we only wish to recall that, at present, the best result of this type is (see \cite[Main Theorem]{BDV}) 
\begin{equation}
\label{bradepvel}
|\Omega|^\frac{2}{N}\,\lambda(\Omega)- |B|^\frac{2}{N}\,\lambda(B)\ge C\,\mathcal{A}(\Omega)^2,
\end{equation}
where $C=C(N)>0$ and $\mathcal{A}$ is the so-called {\it Fraenkel asymmetry}, defined by
\[
\mathcal{A}(\Omega)=\inf\left\{\frac{|\Omega\Delta B|}{|\Omega|}\, :\, B \mbox{ is a ball with }|B|=|\Omega| \right\}.
\]
The symbol $\Omega \Delta B$ stands for the symmetric difference of the relevant sets.
\par
Observe that the quantitative Faber-Krahn inequality \eqref{bradepvel} gives an $L^1$ control on how far $\Omega$ is from being a ball, in terms of how far $\Omega$ is from attaining equality in \eqref{faber}. Moreover, we recall that \eqref{bradepvel} {\it is sharp}, in the sense that the exponent $2$ on the asymmetry can not be lowered.
\subsection{The fractional case}
The main goal of this work is to investigate the same kind of question for the {\it fractional Laplacian of order $s$}, where $0<s<1$. This operator, which eventually became quite popular in the last years, is defined by
\[
(-\Delta)^s u(x)=\lim_{\varepsilon\searrow 0}\int_{\mathbb{R}^N\setminus B_\varepsilon(x)}\frac{u(x)-u(y)}{|x-y|^{N+2\,s}}\,dy.
\]
The usual Laplacian operator is (formally) recovered in the limit as $s\nearrow 1$, up to a suitable rescaling.
\par
For our purposes, it is important to remark that such a linear operator has a variational nature. Indeed, it arises as the first variation of the nonlocal quadratic functional
\[
u\mapsto [u]^2_{W^{s,2}(\mathbb{R}^N)}:=\iint_{\mathbb{R}^N\times \mathbb{R}^N} \frac{|u(x)-u(y)|^2}{|x-y|^{N+2\,s}}\,dx\,dy.
\]
\begin{oss}[Limiting cases]
\label{oss:limit}
It is noteworthy to recall that the nonlocal quantity $[\,\cdot\,]^2_{W^{s,2}(\mathbb{R}^N)}$ has an interpolative nature, i.e. it can be thought as a real interpolation with parameter $s$ of the two quantities
\[
\int |u|^2\,dx\qquad \mbox{ and }\qquad \int |\nabla u|^2\,dx.
\]
Then it is natural to expect that
\[
[u]^2_{W^{s,2}(\mathbb{R}^N)}\sim \frac{C}{s}\,\int |u|^2\,dx,\qquad\mbox{ for } s\searrow 0,
\]
and
\[
[u]^2_{W^{s,2}(\mathbb{R}^N)}\sim \frac{C}{1-s}\,\int |\nabla u|^2\,dx,,\qquad\mbox{ for } s\nearrow 1.
\]
This can be made rigourous, see \cite{MS} for the first result and \cite{bourgain} for the second one.
\end{oss}
The {\it first eigenvalue of the fractional Dirichlet-Laplacian of order $s$ on} $\Omega$ is defined as the smallest real number $\lambda$ such that the following boundary value problem
\[
\left\{\begin{array}{rccl}
(-\Delta)^s u&=&\lambda\,u,& \mbox{ in }\Omega,\\
u&=& 0, & \mbox{ in } \mathbb{R}^N\setminus\Omega,
\end{array}
\right.
\]
admits a nontrivial solution $u\in \mathcal{D}^{s,2}_0(\Omega)$. In analogy with the local case, this space is defined as the completion of $C^\infty_0(\Omega)$ with respect to the norm $[\,\cdot\,]_{W^{s,2}(\mathbb{R}^N)}$. We will indicate the first eigenvalue by $\lambda_s(\Omega)$, while a nontrivial solution $u$ will be called a {\it first eigenfunction} for $\Omega$. 
\par
Observe that the operator $(-\Delta)^s$ is nonlocal in nature. Accordingly, the boundary values are prescribed in a nonlocal sense, as well. 
\par
It is not difficult to see that the first eigenvalue has the following variational characterization
\[
\lambda_s(\Omega)=\min_{u\in \mathcal{D}^{s,2}_0(\Omega)} \Big\{[u]_{W^{s,2}(\mathbb{R}^N)}^2\, \, :\, \|u\|_{L^2(\Omega)}=1\}.
\]
Then, as in the case of the Laplacian previously discussed, one can use symmetrization techniques and prove the following {\it fractional Faber-Krahn inequality} (see for example \cite[Theorem 3.5]{BLP})
\begin{equation}
\label{FK}
|\Omega|^\frac{2\,s}{N}\,\lambda_s(\Omega)\ge |B|^\frac{2\,s}{N}\,\lambda_s(B),
\end{equation}
where $B$ is any $N-$dimensional ball. The proof is the same as in the local case, but in place of \eqref{PSintro} one has to use the {\it nonlocal P\'olya-Szeg\H{o} principle}
\begin{equation}
\label{PS}
[u]^2_{W^{s,2}(\mathbb{R}^N)}\ge [u^*]^2_{W^{s,2}(\mathbb{R}^N)},
\end{equation}
proved in \cite[Theorem 9.2]{AL}.
Moreover by using the characterization of equality cases in \eqref{PS} (see \cite[Theorem A.1]{FS}), one can also characterize balls as the unique sets giving the equality sign in \eqref{FK}. 
\begin{oss}[Other proofs]
As in the case of the Laplacian, for $(-\Delta)^s$ it is possible to adopt a probabilistic point of view, as well. Accordingly, it is possible to give a proof of the fractional Faber-Krahn inequality by using probabilistic techniques, see \cite[Theorem 5]{BLM}. In a PDEs-friendly language, the proof of \cite{BLM} is based on the following idea: if one considers the solution $u_\Omega$ to the following nonlocal diffusion problem
\[
\left\{\begin{array}{rccl}
-(-\Delta)^s u&=&u_t,& \mbox{ in }\Omega\times(0,+\infty),\\
u&=& 0, & \mbox{ in } (\mathbb{R}^N\setminus\Omega)\times[0,+\infty),\\
u(0,\cdot)&=&1,& \mbox{ in } \Omega,
\end{array}
\right.
\]
one can prove that
\begin{equation}
\label{luttinger}
u_\Omega(x,t)\le u_{B_\Omega}(0,t),\qquad \mbox{ for } t>0.
\end{equation}
As before, $B_\Omega$ is the ball centered at the origin, such that $|\Omega|=|B_\Omega|$. By using this pointwise bound and the long-time behavior
\[
u_\Omega(x,t)\sim C\,e^{-\lambda_s(\Omega)\,t},\qquad \mbox{ for } t\to +\infty,
\] 
we get the Faber-Krahn inequality by taking the logarithm on both sides of \eqref{luttinger} and passing to the limit as $t$ goes to $+\infty$.
\par
We also wish to mention the alternative proof of \cite[Theorem 6.1]{SVV}, which is quite close in spirit to that of \cite{BLM}.
\end{oss}
The question we want to address in this paper is the following one: {\it is it possible to add a remainder term in \eqref{FK}, in such a way that the {\rm deficit}
\[
|\Omega|^\frac{2\,s}{N}\,\lambda_s(\Omega)- |B|^\frac{2\,s}{N}\,\lambda_s(B),
\]
controls the lack of spherical symmetry of $\Omega$?}
\subsection{Main result}
We give a positive answer to this question. Actually, at the same price, we can treat a more general family of Faber-Krahn inequalities. In order to present our main result, let us introduce some further notation.
\par
For $N\ge 2$ and $0<s<1$, we set
\[
2^*_s=\frac{2\,N}{N-2\,s}.
\]
Then for every $1\le q<2^*_s$, we consider the sharp Poincar\'e-Sobolev constant
\[
\lambda_{s,q}(\Omega)=\min_{u\in \mathcal{D}^{s,2}_0(\Omega)} \Big\{[u]_{W^{s,2}(\mathbb{R}^N)}^2\, :\, \|u\|_{L^q(\Omega)}=1\Big\}.
\]
The particular case $q=2$ coincides with the first eigenvalue of $(-\Delta)^s$ defined above. For $q\not=2$, any solution of the variational problem above solves the following semilinear problem 
\[
\left\{\begin{array}{rlll}
(-\Delta)^s u&=&\lambda_{s,q}(\Omega)\,|u|^{q-2}\,u,& \mbox{ in }\Omega,\\
u&=& 0, & \mbox{ in } \mathbb{R}^N\setminus\Omega.
\end{array}
\right.
\]
By using \eqref{PS}, one immediately gets
a Faber-Krahn inequality for this quantity, i.e.
\[
|\Omega|^{\frac{2}{q}-1+\frac{2\,s}{N}}\,\lambda_{s,q}(\Omega)\ge |B|^{\frac{2}{q}-1+\frac{2\,s}{N}}\,\lambda_{s,q}(B).
\] 
The main result of this paper is the following one. 
\begin{teo}
\label{teo:main}
Let $N\ge 2$, $0<s<1$ and $1\le q<2^*_s$. For every $\Omega\subset\mathbb{R}^N$ open set with finite measure, we have
\[
|\Omega|^{\frac{2}{q}-1+\frac{2\,s}{N}}\,\lambda_{s,q}(\Omega)-|B|^{\frac{2}{q}-1+\frac{2\,s}{N}}\,\lambda_{s,q}(B)\ge \frac{\sigma_1}{(1-s)}\,\mathcal{A}(\Omega)^\frac{3}{s},
\]
for an explicit constant $\sigma_1=\sigma_1(N,s,q)>0$, which is uniform as $s\nearrow 1$.
\end{teo}
\begin{oss}[Limit cases]
\label{oss:limits}
By keeping in mind Remark \ref{oss:limit}, it is natural to expect that for $s\nearrow 1$
\[
\lambda_{s,q}(\Omega)\sim \frac{C}{1-s}\,\lambda_{1,q}(\Omega),\qquad \mbox{ where } \lambda_{1,q}(\Omega)=\min_{u\in \mathcal{D}^{1,2}_0(\Omega)} \Big\{\|\nabla u\|_{L^2(\Omega)}^2\, :\, \|u\|_{L^q(\Omega)}=1\Big\},
\]
and thus Theorem \ref{teo:main} should give a quantitative Faber-Krahn inequality for the local case, in the limit. 
\par
This is actually the case. More precisely, if we keep $q$ fixed and let $s$ go to $1$ in Theorem \ref{teo:main}, by using the controlled behavior of the constant $\sigma_1$ and Lemma \ref{lm:convergenza} in the Appendix, we end up with the quantitative Faber-Krahn inequality for the Laplacian
\[
|\Omega|^{\frac{2}{q}-1+\frac{2}{N}}\,\lambda_{1,q}(\Omega)-|B|^{\frac{2}{q}-1+\frac{2}{N}}\,\lambda_{1,q}(B)\ge C\, \mathcal{A}(\Omega)^3.
\]
The latter has been already proved by the first author and De Philippis in \cite[Theorem 2.10]{BD}, by adapting the idea of Hansen and Nadirashvili contained in \cite{HN}.
\par
On the other hand, if we keep $0<s<1$ fixed and let $q$ go to $2^*_s$, by Lemma \ref{lm:towardsS} we get
\[
\lim_{q\nearrow 2^*_s} \left(|\Omega|^{\frac{2}{q}-1+\frac{2\,s}{N}}\,\lambda_{s,q}(\Omega)-|B|^{\frac{2}{q}-1+\frac{2\,s}{N}}\,\lambda_{s,q}(B)\right)=0,
\]
which shows that 
\[
\lim_{q\nearrow 2^*_s} \sigma_1=0.
\]
\end{oss}
Apart for the case $q=2$, also the case $q=1$ deserves to be singled out. In analogy with the local case, we call the quantity
\[
\mathcal{T}_s(\Omega):=\frac{1}{\lambda_{s,1}(\Omega)}=\max_{u\in \mathcal{D}^{s,2}_0(\Omega)} \left\{\left(\int_\Omega |u|\,dx\right)^2\, :\, [u]^2_{W^{s,2}(\mathbb{R}^N)}=1\right\},
\]
{\it fractional torsional rigidity of order $s$} of $\Omega$. It is not difficult to see that 
\[
\mathcal{T}_s(\Omega)=\int_\Omega w_{s,\Omega}\,dx,
\] 
where $w_{s,\Omega}$ is called {\it $s-$torsion function of $\Omega$} and is the unique solution to the boundary value problem
\[
\left\{\begin{array}{rccl}
(-\Delta)^s u&=&1,& \mbox{ in }\Omega,\\
u&=& 0, & \mbox{ in } \mathbb{R}^N\setminus\Omega.
\end{array}
\right.
\]
We refer to \cite{Fr} for a detailed study of some interesting features of this function.
As a straightforward consequence of Theorem \ref{teo:main}, we obtain the following
\begin{coro}
\label{coro:torsion}
Let $N\ge 2$ and  $0<s<1$. For every $\Omega\subset\mathbb{R}^N$ open set with finite measure, we have
\[
\frac{\mathcal{T}_s(B)}{|B|^\frac{N+2\,s}{N}}-\frac{\mathcal{T}_s(\Omega)}{|\Omega|^\frac{N+2\,s}{N}}\ge \sigma_2\,(1-s)\,\mathcal{A}(\Omega)^\frac{3}{s},
\]
for an explicit constant $\sigma_2=\sigma_2(N,s)>0$, which is uniform as $s\nearrow 1$.
\end{coro}

\subsection{Strategy of the proof}

For ease of presentation, we now stick to the case $q=2$.
The first naive idea would be to try and insert quantitative elements in the nonlocal P\'olya-Szeg\H{o} principle \eqref{PS}. Already in the local case, this idea is quite complicate to implement and proofs exploiting this route usually produce stability estimates with non-sharp exponents on the Fraenkel asymmetry (see for example \cite{FMP2, Po, Sz}). At present, the best estimate obtained in this way is
\begin{equation}
\label{HN2}
|\Omega|^\frac{2}{N}\,\lambda(\Omega)-|B|^\frac{2}{N}\,\lambda(B)\ge C\, \mathcal{A}(\Omega)^3,
\end{equation}
which is the result of \cite[Theorem 2.10]{BD} already mentioned in Remark \ref{oss:limits}.
\par
In addition to this, this approach is even more complicate in the nonlocal case, due to the absence of a true Coarea Formula for nonlocal integrals. Indeed, the proof of \eqref{PS} is based on the {\it Riesz's rearrangement inequality}, whose identification of equality cases is quite subtle (see \cite{Bu}).
\par
Thus, the first step is to give another proof of the Faber-Krahn inequality, which circumvents the nonlocality of the problem. This is done by adding one extra variable $z$ and considering a suitable extension problem in the upper half-space $\{(x,z)\in\mathbb{R}^N\times\mathbb{R}\, :\, z>0\}$.
Since the appearing of the paper \cite{CS}, this procedure has become standard in the field. 
\par
In the context of stability estimates for nonlocal energies, this idea has been previously employed by Fusco, Millot and Morini in their paper \cite{FMM}. In the latter, the authors proved a quantitative stability estimate for the {\it fractional isoperimetric inequality of order $s$}, i.e.
\[
|\Omega|^\frac{s-N}{N}\,P_s(\Omega)\ge |B|^\frac{s-N}{N}\,P_s(B),
\]
where $B$ is a ball and $P_s$ stands for the $s-$perimeter of a set, defined by
\[
P_s(\Omega)=[1_\Omega]^2_{W^{\frac{s}{2},2}(\mathbb{R}^N)}=2\,\iint_{\Omega\times (\mathbb{R}^N\setminus\Omega)} \frac{1}{|x-y|^{N+s}}\,dx\,dy.
\]
In order to give a better understanding of our strategy, we give a sketch of the proof of the fractional Faber-Krahn inequality by using this extension procedure. We refer to Section \ref{sec:3} for more details.
\par
Given a first eigenfunction $u$ for $\Omega$ with unit $L^2$ norm, we know that
\begin{equation}
\label{1}
\lambda_s(\Omega)=[u]^2_{W^{s,2}(\mathbb{R}^N)}=C\, \iint_{\mathbb{R}^N\times \mathbb{R}_+} z^{1-2\,s}\,|\nabla U|^2\,dx\,dz,
\end{equation}
where $U$ is the unique solution of the following variational problem
\[
\min \left\{\iint_{\mathbb{R}^N\times \mathbb{R}_+} z^{1-2\,s}\,|\nabla V|^2\,dx\,dz\, :\, V=u \mbox{ on } \{z=0\} \right\},
\]
and $C>0$ is a universal constant.
By making a slight abuse of notation and indicating by $U^*$ the Schwarz symmetrization of $U$ with respect to the variable $x$, we have as in \cite{FMM}
\begin{equation}
\label{2}
\iint_{\mathbb{R}^N\times \mathbb{R}_+} z^{1-2\,s}\,|\nabla_x U|^2\,dx\,dz\ge \iint_{\mathbb{R}^N\times \mathbb{R}_+} z^{1-2\,s}\,|\nabla_x U^*|^2\,dx\,dz,
\end{equation}
and
\begin{equation}
\label{2bis}
\iint_{\mathbb{R}^N\times \mathbb{R}_+} z^{1-2\,s}\,|\partial_z U|^2\,dx\,dz\ge \iint_{\mathbb{R}^N\times \mathbb{R}_+} z^{1-2\,s}\,|\partial_z U^*|^2\,dx\,dz.
\end{equation}
Moreover, $U^*$ coincides with $u^*$ on the boundary $\{z=0\}$. Thus we get
\[
\iint_{\mathbb{R}^N\times \mathbb{R}_+} z^{1-2\,s}\,|\nabla U^*|^2\,dx\,dz\ge \min \left\{\iint_{\mathbb{R}^N\times \mathbb{R}_+} z^{1-2\,s}\,|\nabla V|^2\,dx\,dz\, :\, V=u^* \mbox{ on } \{z=0\} \right\},
\]
so that
\begin{equation}
\label{3}
C\,\iint_{\mathbb{R}^N\times \mathbb{R}_+} z^{1-2\,s}\,|\nabla U^*|^2\,dx\,dz\ge [u^*]^2_{W^{s,2}(\mathbb{R}^N)}.
\end{equation}
By observing that $u^*$ is admissible for the variational problem which defines $\lambda_s(B_\Omega)$, we can now get the fractional Faber-Krahn inequality by combining \eqref{1}, \eqref{2}, \eqref{2bis} and \eqref{3}.
\par
In order to prove the quantitative statement of Theorem \ref{teo:main}, the idea is now to insert quantitative elements in the proof of \eqref{2}. We will follow the ideas of Hansen and Nadirashvili, from their above mentioned paper \cite{HN}.
By using the Coarea Formula and the {\it sharp quantitative isoperimetric inequality} (see \cite{FMP}), we can proceed as in the local case of \cite[Theorem 2.10]{BD}. This leads to a quantitative enhancement of the form
\[
\lambda_s(\Omega)-\lambda_s(B_\Omega)\gtrsim \int_{\mathbb{R}_+} z^{1-2\,s}\left(\int_0^{+\infty} \mathcal{A}(E_{t,z})^2\,dt\right)\,dz,
\]
where $E_{t,z}=\{x\in\mathbb{R}^N\,:\, U(x,z)>t\}$ are the ``horizontal'' level sets of the extension $U$. There is now a twofold difficulty: at first, we have to relate the asymmetry of this ``artificial'' level sets to those of the first eigenfuction $u$, i.e. $\Omega_t=\{x\in \Omega\, :\, u(x)>t\}$. In other words, we wish to prove something of the type
\[
\mathcal{A}(E_{t,z})\simeq \mathcal{A}(\Omega_t),\qquad \mbox{ for } t\ll 1 \mbox{ and } z\ll 1.
\]
Secondly, we need to relate all these asymmetries to that of $\Omega$, i.e. the zero level set of $u$. On the other hand, in this process particular attention should be put in avoiding the zero level set of the extension $U$: indeed, by the minimum principle this would coincide with the whole $\mathbb{R}^N$ and the information on the propagation of the asymmetry would be completely lost.
\begin{oss}[Sharpness]
We do not expect our estimate to be sharp. Indeed, it is natural to conjecture that Theorem \ref{teo:main} should hold with $\mathcal{A}(\Omega)^2$ in place of $\mathcal{A}(\Omega)^{3/s}$.
\par
We point out that, already in the local case $s=1$, the sharp quantitative Faber-Krahn inequality of \cite{BDV} comes with an {\it unknown} stability constant. Indeed, the method of proof is based on the so-called {\it selection principle} and is not constructive.
\par
 At present, for $s=1$ the best result with an explicit constant is \eqref{HN2}, where the Fraenkel asymmetry has an exponent $3$.
Then our result can be seen as the natural fractional counterpart of this last result.
\end{oss}

\subsection{Plan of the paper}
In Section \ref{sec:2} we settle all the definitions and the machinery needed in the sequel of the paper. In particular, we introduce the extension problem to the half-space $\mathbb{R}^N\times\mathbb{R}_+$. We show in Section \ref{sec:3} how to exploit this extension problem in order to prove the fractional Faber-Krahn inequality. 
\par
We then pass to consider the stability issue: at this aim, we need some technical results about the propagation of asymmetry from the set $\Omega$ to the ``horizontal'' level sets of the solution of the extension problem. This is the content of Section \ref{sec:4}. 
\par
We eventually prove our main result Theorem \ref{teo:main} in Section \ref{sec:5}. Then in Section \ref{sec:6} we briefly show how it is possible to improve our stability exponent $3/s$ with the same method, provided the sets considered are smoother (Theorem \ref{teo:smooth}).
\par
The paper ends with two appendices, aimed at proving some technical results.

\begin{ack}
E.\,C. has been supported by MINECO grants MTM2014-52402-C3-1-P and MTM2017-84214-C2-1-P, and is part of the Catalan research group 2014 SGR 1083. E.\,C. is a member of the Gruppo Nazionale per l'Analisi Matematica, la Probabilit\`a
e le loro Applicazioni (GNAMPA) of the Istituto Nazionale di Alta Matematica (INdAM).
\end{ack}

\section{Preliminaries}
\label{sec:2}

\subsection{Fractional Sobolev spaces}
Let $0<s<1$, for a measurable function $u:\mathbb{R}^N\to \mathbb{R}$ we define
\[
[u]_{W^{s,2}(\mathbb{R}^N)}=\left(\iint_{\mathbb{R}^N\times\mathbb{R}^N} \frac{|u(x)-u(y)|^2}{|x-y|^{N+2\,s}}\,dx\,dy\right)^\frac{1}{2}.
\]
Accordingly, we consider the {\it Sobolev-Slobodecki\u{\i} space}
\[
W^{s,2}(\mathbb{R}^N)=\Big\{u\in L^2(\mathbb{R}^N)\, :\, [u]_{W^{s,2}(\mathbb{R}^N)}<+\infty\Big\}.
\]
It is a classical fact that 
\[
W^{s,2}(\mathbb{R}^N)=H^s(\mathbb{R}^N)=\left\{F\in\mathcal{S}'(\mathbb{R}^N)\, :\, \int_{\mathbb{R}^N}(1+|\xi|^2)^s\,|\widehat{F}(\xi)|^2\,d\xi<+\infty \right\},
\]
with the usual notation
\[
\widehat \varphi(\xi)=\frac{1}{(2\,\pi)^{N/2}}\,\int_{\mathbb{R}^N}\varphi(x)\,e^{-i\,\xi\cdot x}\,dx,
\]
for the Fourier transform.
\par
For $N\ge 2$, by the {\it fractional Sobolev inequality} we have the continuous inclusion
\[
W^{s,2}(\mathbb{R}^N)\subset L^{2^*_s}(\mathbb{R}^N),
\]
thus, by duality, we get the following continuous inclusion for the topological dual spaces
\begin{equation}
\label{dualedentro}
L^{(2^*_s)'}(\mathbb{R}^N)\subset (W^{s,2}(\mathbb{R}))^*=H^{-s}(\mathbb{R}^N).
\end{equation}
For any open set $\Omega\subset\mathbb{R}^N$, we define the {\it homogeneous Sobolev-Slobodecki\u{\i} space} $\mathcal{D}^{s,2}_0(\Omega)$ as the completion of $C^\infty_0(\Omega)$ with respect to the norm
\[
u\mapsto [u]_{W^{s,2}(\mathbb{R}^N)}.
\] 
Observe that the latter is indeed a norm on $C^\infty_0(\Omega)$.
\par
For $N\ge 2$, by the fractional Sobolev inequality we have that $\mathcal{D}^{s,2}_0(\Omega)$ is always a functional space, such that
\[
\mathcal{D}^{s,2}_0(\Omega)\subset \mathcal{D}^{s,2}_0(\mathbb{R}^N)\subset L^{2^*_s}(\mathbb{R}^N),
\]
with continuous inclusions.
\begin{lm}
\label{lm:eigenfunction}
Let $N\ge 2$, $0<s<1$ and $1\le q<2^*_s$. For every $\Omega\subset\mathbb{R}^N$ open bounded set, we have
\[
\lambda_{s,q}(\Omega)=\inf_{u\in C^\infty_0(\Omega)} \left\{[u]_{W^{s,2}(\mathbb{R}^N)}^2\, :\, \int_\Omega |u|^q\,dx=1\right\}>0.
\]
Moreover, the infimum above is attained by a function $u_\Omega\in\mathcal{D}^{s,2}_0(\Omega)\cap L^\infty(\Omega)\cap C^\beta(\Omega)$ with $0<\beta<\min\{2\,s,1\}$ and such that $u_\Omega>0$ in $\Omega$. 
\end{lm}
\begin{proof}
The compactness of the embedding $\mathcal{D}^{s,2}_0(\Omega)\hookrightarrow L^q(\Omega)$ (see for example \cite[Corollary 2.8]{BLP}) entails that $\lambda_{s,q}(\Omega)>0$ and that there exists a minimizer $u_\Omega\in\mathcal{D}^{s,2}_0(\Omega)$. The fact that we can choose $u_\Omega$ to be non-negative follows from the fact that
\[
[|u|]_{W^{s,2}(\mathbb{R}^N)}\le [u]_{W^{s,2}(\mathbb{R}^N)}.
\]
Such a minimizer is a non-negative weak solution of 
\[
(-\Delta)^s u=\lambda_{s,q}(\Omega)\,u^{q-1},\qquad \mbox{ in } \Omega.
\]
By using \cite[Theorem 3.2]{IMS}, we have that $u_\Omega\in L^\infty(\Omega)$. The claimed continuity of $u_\Omega$ then follows from \cite[Theorem 1.4]{BLS}, for example.
Finally, we have $u_\Omega>0$ in $\Omega$ by the minimum principle. 
\end{proof}

\begin{comment}
\begin{teo}
\label{teo:Linfty}
Let $N\ge 2$ and $0<s<1$. Let $\Omega\subset\mathbb{R}^N$ be an open bounded set. We have
\[
\|u_\Omega\|_{L^\infty(\Omega)}\le C_{N,s}\,\Big(\lambda_{s,q}(\Omega)\Big)^\gamma,
\]
for a constant $C_{N,s}>0$ such that
\[
C_{N,s}\sim (1-s) \quad \mbox{ for }s\nearrow 1, 
\]
and
\[
C_{N,s}\sim s\quad \mbox{ for }s\searrow 0.
\]
\end{teo}
\begin{proof}
{\color{red} Usare iterazione di Moser come in \cite[Theorem 3.3]{BLP} oppure una iterazione di De Giorgi come in Franzina-Palatucci. Ci servir\`a un controllo uniforme della stima per $s\nearrow 1$, attenzione al caso $N=2$ che diventa ``conforme'' nel limite $s\nearrow 1$! Si veda \cite[Theorem 2.10]{BPS}.}
\end{proof}
\end{comment}
The next simple result will be useful.
\begin{lm}
\label{lm:foul2}
Let $1\le p<2$ and
\[
0\le \tau<N\,\left(\frac{1}{p}-\frac{1}{2}\right).
\]
Then we have the continuous inclusion
\[
L^p(\mathbb{R}^N)\cap L^2(\mathbb{R}^N)\subset H^{-\tau}(\mathbb{R}^N).
\]
More precisely, for every $u\in L^p(\mathbb{R}^N)\cap L^2(\mathbb{R}^N)$, we have
\begin{equation}
\label{foul2}
\left(\int_{\mathbb{R}^N} |\xi|^{-2\,\tau}\,|\widehat u(\xi)|^2\,d\xi\right)^{1/2}\leq C\,\Big(\|u\|_{L^p(\mathbb{R}^N)}\Big)^{\frac{2\,p}{2-p}\,\frac{\tau}{N}}\,\Big(\|u\|_{L^2(\mathbb{R}^N)}\Big)^{1-\frac{2\,p}{2-p}\,\frac{\tau}{N}},
\end{equation}
for a constant $C=C(N,\tau,p)>0$, which blows-up as $\tau\nearrow N\,(2-p)/(2\,p)$.
\end{lm}
\begin{proof}
The assumption $u\in L^p(\mathbb{R}^N)\cap L^2(\mathbb{R}^N)$ entails that $\widehat u\in L^2(\mathbb{R}^N)\cap L^{p'}(\mathbb{R}^N)$. Moreover, we have
\[
\|\widehat u\|_{L^{p'}(\mathbb{R}^N)}\le C_{N,p}\,\|u\|_{L^p(\mathbb{R}^N)}\qquad \mbox{ and }\qquad \|\widehat u\|_{L^2(\mathbb{R}^N)}=\|u\|_{L^2(\mathbb{R}^N)}.
\] 
Hence, by fixing $\lambda>0$ we obtain
\[
\begin{split}
\left(\int_{\mathbb{R}^N}|\xi|^{-2\,\tau}\,|\widehat u(\xi)|^2\,d\xi\right)^{1/2}&\le \left(\int_{\{|\xi|<\lambda\}} |\xi|^{-2\,\tau}\,|\widehat u(\xi)|^2\,d\xi\right)^{1/2}+\left(\int_{\{|\xi|\ge \lambda\}} |\xi|^{-2\,\tau}\,|\widehat u(\xi)|^2\,d\xi\right)^{1/2}\\
&\leq \|\widehat u\|_{L^{p'}(\mathbb{R}^N)}\left(\int_{\{|\xi|<\lambda\}} |\xi|^{-\tau\,\frac{2\,p'}{p'-2}}\,d\xi\right)^\frac{p'-2}{2\,p'}+\lambda^{-\tau}\,\|\widehat u\|_{L^2(\mathbb{R}^N)}\\
&\leq C\,\lambda^{N\,\frac{p'-2}{2\,p'}-\tau}\,\|u\|_{L^p(\mathbb{R}^N)}+\lambda^{-\tau}\,\|u\|_{L^2(\mathbb{R}^N)}.
\end{split}
\]
Observe that 
\[
N\,\frac{p'-2}{2\,p'}-\tau=N\,\left(\frac{1}{p}-\frac{1}{2}\right)-\tau>0,
\]
then by taking the minimum over $\lambda>0$, we get the desired conclusion \eqref{foul2}. 
\end{proof}
\subsection{The extension problem}
We set $\mathbb{R}^{N+1}_+:=\mathbb{R}^N\times\mathbb{R}_+$ and denote by $(x,z)$ the points in $\mathbb{R}^{N+1}_+$, i.e. $x\in \mathbb{R}^N$ and $z>0$. We now define the Sobolev space that will be exploited for our purposes.
\begin{defi}
Let $N\ge 2$ and $0<s<1$. We define the {\it weighted Sobolev space} $\mathcal{H}^{1,s}(\mathbb{R}^{N+1}_+)$ as
$$\mathcal{H}^{1,s}(\mathbb{R}^{N+1}_+)=\Big\{U:\mathbb{R}^{N+1}_+\to\mathbb{R} \, : \, Uz^{\frac{1-2s}{2}}\in L^2(\mathbb{R}^{N+1}_+) \ \mbox{ and }\  |\nabla U|\,z^{\frac{1-2s}{2}}\in L^2(\mathbb{R}^{N+1}_+)\Big\}.$$
We endow such a space with the norm
$$\|U\|_{\mathcal{H}^{1,s}(\mathbb{R}^{N+1}_+)}=\left(\iint_{\mathbb{R}^{N+1}_+} z^{1-2\,s}\,\Big(|U|^2+|\nabla U|^2\Big)\,dx\,dz\right)^\frac{1}{2}.$$
\end{defi}
We need to consider traces of functions in the previous space. The following result is a trace theorem for $\mathcal{H}^{1,s}(\mathbb{R}^{N+1}_+)$. Recall that $\partial\,\mathbb{R}^{N+1}_+=\{(x,z)\, :\, z=0\}\simeq \mathbb{R}^N$. 
\begin{lm}[Trace space]
\label{prop:space}
Let $N\ge 2$ and $0<s<1$. There exists a linear and continuous trace operator
\[
\mathrm{trace}:\mathcal{H}^{1,s}(\mathbb{R}^{N+1}_+)\to W^{s,2}(\mathbb{R}^N),
\]
which is surjective. Moreover, the closed subspace 
\[
\mathcal{H}^{1,s}_0(\mathbb{R}^{N+1}_+)=\Big\{U\in \mathcal{H}^{1,s}(\mathbb{R}^{N+1}_+)\ : \ \mathrm{trace}(U)=0\Big\},
\]
coincides with the closure of $C^\infty_0(\mathbb{R}^{N+1}_+)$ in $\mathcal{H}^{1,s}(\mathbb{R}^{N+1}_+)$.
\end{lm}
\begin{proof}
By using \cite[Section 5]{Lio}, we know that there exists a linear, continuous and surjective operator 
\[
\mathrm{trace}: \mathcal{H}^{1,s}(\mathbb{R}^{N+1}_+)\to W^{s,2}_\diamond(\mathbb{R}^N),
\]
where 
\[
W^{s,2}_\diamond(\mathbb{R}^N)=\left\{u\in L^2(\mathbb{R}^N)\, :\, \sum_{i=1}^N \int_{\mathbb{R}^N}\int_0^{+\infty}\frac{|u(x+\varrho\,\mathbf{e}_i)-u(x)|^2}{\varrho^{1+2\,s}}\,d\varrho\,dx<+\infty \right\}.
\]
By using Proposition \ref{prop:equiRN}, we get for every $u\in C^\infty_0(\mathbb{R}^N)$
\[
\frac{1}{C}\,[u]_{W^{s,2}(\mathbb{R}^N)}^2\le \sum_{i=1}^N \int_{\mathbb{R}^N}\int_0^{+\infty}\frac{|u(x+\varrho\,\mathbf{e}_i)-u(x)|^2}{\varrho^{1+2\,s}}\,d\varrho\,dx\le C\,[u]_{W^{s,2}(\mathbb{R}^N)}^2.
\]
By density, this in turn implies that $W^{s,2}_\diamond(\mathbb{R}^N)=W^{s,2}(\mathbb{R}^N)$.
\par
The proof of the second statement can be done as in \cite[Theorem 5.1, point iii)]{bn}, which deals with the case $s=1/2$. We leave the details to the reader. 
\end{proof}

We now set
\begin{equation}
\label{beta}
P_1(x)=\frac{\beta_{N,s}}{(1+|x|^2)^\frac{N+2\,s}{2}},\qquad \mbox{ where }\ \beta_{N,s}=\left(\int_{\mathbb{R}^N} \frac{1}{(1+|x|^2)^\frac{N+2\,s}{2}}\,dx\right)^{-1},
\end{equation}
and for every $z>0$, we consider the rescaled function
\[
P_z(x)=\frac{1}{z^N}\,P_1\left(\frac{x}{z}\right)=\beta_{N,s}\,\frac{z^{2\,s}}{(z^2+|x|^2)^\frac{N+2\,s}{2}}.
\]
\begin{oss}[The Fourier side of $P_1$]
We observe that $P_1\in L^1(\mathbb{R}^N)\cap L^\infty(\mathbb{R}^N)$ for every $0<s<1$ and that
\[
x\,P_1\in L^\alpha(\mathbb{R}^N)\cap L^\infty(\mathbb{R}^N),\qquad \mbox{ where } \left\{\begin{array}{cc}
\alpha=1,& \mbox{ if } s>1/2,\\
&\\
\alpha>\dfrac{N}{N+2\,s-1},& \mbox{ if } 0<s\le 1/2.
\end{array}
\right.
\]
In particular, by using Lemma \ref{lm:foul2} and the properties of the Fourier transform, we get
\begin{equation}
\label{op}
\int_{\mathbb{R}^N} |\xi|^{2-N-2\,s}\,|\widehat{P}_1(\xi)|^2\,d\xi<+\infty,\qquad \mbox{ for } 0<s<1,
\end{equation}
and\footnote{For $0<s\le 1/2$, we use Lemma \ref{lm:foul2} with $p=\alpha$ given above. We observe that in this case
\[
N\,\frac{2-\alpha}{2\,\alpha}>\frac{N+2\,s-2}{2}\qquad \Longleftrightarrow\qquad \alpha<\frac{N}{N+s-1},
\]
which is feasible, by recalling the limitation on $\alpha$.}
\begin{equation}
\label{la}
\int_{\mathbb{R}^N} |\xi|^{2-N-2\,s}\,|\nabla \widehat{P}_1(\xi)|^2\,d\xi=\int_{\mathbb{R}^N} |\xi|^{2-N-2\,s}\,|\widehat{x\,P_1}|^2\,d\xi<+\infty, \qquad \mbox{ for } 0<s<1.
\end{equation}
\end{oss}
We recall that, as established in \cite{CS}, $P_z$ is the Poisson kernel for the Dirichlet problem \eqref{cafsil} below. 
Indeed, for any given $\varphi\in W^{s,2}(\mathbb{R}^N)$, let us denote by $U_\varphi$ the function on $\mathbb{R}^{N+1}_+$ defined by
\begin{equation}\label{poiss}
U_\varphi(x,z)=P_z\ast \varphi (x)=\int_{\mathbb{R}^{N}} \frac{1}{z^N}\,P_1\left(\frac{y}{z}\right)\,\varphi(x-y)\,dy,\qquad (x,z)\in\mathbb{R}^{N+1}_+.
\end{equation}
Then, $U_\varphi$ is a solution to the following boundary value problem
\begin{equation}\label{cafsil}
\left\{\begin{array}{rcll}
-\mathrm{div}(z^{1-2\,s}\,\nabla U_\varphi)&=&0&\mbox{ in }\mathbb{R}^N\times \mathbb{R}_+,\\
U_\varphi(\cdot,0)&=&\varphi&\mbox{ in } \mathbb{R}^N.
\end{array}
\right.
\end{equation}
As such, it verifies the following weak formulation
\begin{equation}
\label{weakeq}
\iint_{\mathbb{R}^{N+1}_+} z^{1-2\,s}\,\nabla U_\varphi\cdot\nabla\phi\,dx\,dz=0,\qquad\mbox{ for every } \phi\in C^\infty_0(\mathbb{R}^{N+1}_+).
\end{equation}
In what follows, with a slight abuse of notation we will denote by $\widehat U_\varphi$ the partial Fourier transform, taken with respect to the $x-$variable.
\begin{prop}\label{mini}
Let $N\ge 2$  and $0<s<1$. For every 
\[
\varphi\in W^{s,2}(\mathbb{R}^N)\cap L^{(2^*_{1-s})'}(\mathbb{R}^N),
\] 
the following variational problem
\begin{equation}
\label{min}
\min_{U\in \mathcal{H}^{1,s}(\mathbb{R}^{N+1}_+)} \left\{\iint_{\mathbb{R}^{N+1}_+} z^{1-2\,s}\,|\nabla U|^2\,dx\,dz\, :\, \mathrm{trace}(U)=\varphi \right\},
\end{equation}
admits a unique solution, which coincides with $U_\varphi$ given in \eqref{poiss}. Moreover, we have
\begin{equation}
\label{equiv}
[\varphi]_{W^{s,2}(\mathbb{R}^N)}^2=\gamma_{N,s}\iint_{\mathbb{R}^{N+1}_+}z^{1-2\,s}\,|\nabla U_\varphi|^2\, dx\,dz,
\end{equation}
and
\begin{equation}
\label{lemma2}
\|U_\varphi(\cdot,z)-\varphi\|^2_{L^2(\mathbb{R}^N)}\le \Big(\beta_{N,s}\,[\varphi]^2_{W^{s,2}(\mathbb{R}^N)}\Big)\,z^{2\,s},\qquad \mbox{ for a.\,e. }z>0.
\end{equation}
Here $\beta_{N,s}$ is the constant in \eqref{beta} and $\gamma_{N,s}$  is a positive constant whose precise value is given in Remark \ref{oss:constants}.
\end{prop}
\begin{proof}
We first observe that by surjectivity of the trace map given in Lemma \ref{prop:space}, the class of admissible functions in \eqref{min} is not empty. We need to show that $U_\varphi$ is in the relevant Sobolev space, i.e.
\begin{equation}
\label{dentro}
\iint_{\mathbb{R}^{N+1}_+} z^{1-2\,s}\,\left(|U_\varphi|^2+|\nabla U_\varphi|^2\right)\,dx\,dz<+\infty.
\end{equation}
In order to prove this, one can argue as in the proof in \cite{CS} and use the partial Fourier transform. 
Indeed, by \eqref{poiss} we have $\widehat U_\varphi(\xi,z)=\widehat \varphi(\xi)\,\widehat{P}_1(z\,\xi)$ and $\widehat{P}$ is a radial function, i.e.
\[
\widehat{P}_1(z\,\xi)=v(z\,|\xi|),
\] 
for a suitable function $v$. The key point is that 
\[
\mathcal{J}(v):=\int_0^{+\infty}z^{1-2s}\,\left(|v|^2+|v'|^2\right)\,dz<+\infty.
\]
Indeed, by using spherical coordinates and the radial symmetry of $\widehat{P}_1$, we have
\[
\begin{split}
\int_0^{+\infty}z^{1-2s}\,\left(|v|^2+|v'|^2\right)\,dz&=\int_0^{+\infty} z^{2-N-2\,s}\,\left(|v|^2+|v'|^2\right)\,z^{N-1}\,dz\\
&=\fint_{\mathbb{S}^{N-1}}\int_0^{+\infty}z^{2-N-2\,s}\,\left(|v|^2+|v'|^2\right)\,z^{N-1}\,dz\,d\mathcal{H}^{N-1}\\
&=\frac{1}{N\,\omega_N}\,\int_{\mathbb{R}^N} |\xi|^{2-N-2\,s}\,\left(|\widehat{P}_1(\xi)|^2+|\nabla \widehat{P}_1(\xi)|^2\right)\,d\xi,
\end{split}
\]
which is finite, thanks to \eqref{op} and \eqref{la}.
By using Plancherel's identity, we get
\[
\begin{split}
\iint_{\mathbb{R}^{N+1}_+} z^{1-2\,s}\,|U_\varphi|^2\,dx\,dz&=\iint_{\mathbb{R}^{N+1}_+} z^{1-2\,s}\, |\widehat U_\varphi(\xi)|^2\,d\xi\,dz\\
&=\int_{\mathbb{R}^N}\left(\int_{0}^{+\infty} z^{1-2\,s} |v(z\,|\xi|)|^2\,dz\right)\, |\widehat{\varphi}(\xi)|^2\,d\xi\\
&=\int_{\mathbb{R}^N}\left(\int_{0}^{+\infty} t^{1-2\,s} |v(t)|^2\,dt\right)\,|\xi|^{2\,s-2}\, |\widehat{\varphi}(\xi)|^2\,d\xi\\
&\le \mathcal{J}(v)\,\int_{\mathbb{R}^N} |\xi|^{2\,s-2}\,|\widehat{\varphi}(\xi)|^2\,d\xi.
\end{split}
\]
The last integral is finite, thanks to the fact that $\varphi\in L^{(2^*_{1-s})'}(\mathbb{R}^N)$ and recalling \eqref{dualedentro}.
\par
With a similar computation, one can show that $z^{1-2\,s}\,|\nabla U_\varphi|^2 \in L^1(\mathbb{R}^{N+1}_+)$ (see \cite[Section 3.2]{CS}).
This concludes the proof of \eqref{dentro}.
\par
We now show that $U_\varphi$ is a minimizer of our variational problem. Uniqueness then will follows from strict convexity of the functional.
By convexity of the functional, for any admissible function $V$ we have
\[
\iint_{\mathbb{R}^{N+1}_+} z^{1-2\,s}\,|\nabla V|^2\,dx\,dz\geq\iint_{\mathbb{R}^{N+1}_+} z^{1-2\,s}\,|\nabla U_\varphi|^2\,dx\,dz+2\,\iint_{\mathbb{R}^{N+1}_+} z^{1-2\,s}\,\nabla U_\varphi\cdot\nabla(V-U_\varphi)\,dx\,dz.
\]
By using that $U_\varphi-V\in \mathcal{H}^{1,s}_0(\mathbb{R}^{N+1}_+)$ and the density of $C^\infty_0(\mathbb{R}^{N+1}_+)$ in $\mathcal{H}^{1,s}_0(\mathbb{R}^{N+1}_+)$, from \eqref{weakeq} we get
\[
\iint_{\mathbb{R}^{N+1}_+} z^{1-2\,s}\,\nabla U_\varphi\cdot \nabla(V-U_\varphi)\,dx\,dz=0,
\] 
which implies the minimality of $U_\varphi$.
\vskip.2cm\noindent
For equality \eqref{equiv}, we refer to \cite[Theorem 3.1 \& Remark 3.11]{CS1}, where the precise value of the constant $\gamma_{N,s}$ is given (we recall it in Remark \ref{oss:constants} below).
\vskip.2cm\noindent
Finally, in order to prove \eqref{lemma2}, we observe that
\[
\begin{split}
[\varphi]_{W^{s,2}(\mathbb{R}^N)}^2&=\iint_{\mathbb{R}^N\times\mathbb{R}^N} \frac{|\varphi(x)-\varphi(y)|^2}{|x-y|^{N+2\,s}}\,dx\,dy=\int_{\mathbb{R}^N} \left\|\frac{\varphi(\cdot+h)-\varphi(\,\cdot\,)}{|h|^s}\right\|^2_{L^2(\mathbb{R}^N)}\,\frac{dh}{|h|^N},
\end{split}
\]
which follows with a simple change of variable.
By using Minkowski's and H\"older's inequalities, we have
\[
\begin{split}
\|U_\varphi(\cdot,z)-\varphi\|_{L^2(\mathbb{R}^N)}&=\left\|\int_{\mathbb{R}^{N}} P_z(y)\,[\varphi(\cdot-y)-\varphi(\cdot)]\,dy\right\|_{L^2(\mathbb{R}^N)}\\
&\le\int_{\mathbb{R}^N} P_z(y)\,\|\varphi(\cdot-y)-\varphi(\cdot)\|_{L^2(\mathbb{R}^N)}\,dy\\
&\le \left(\int_{\mathbb{R}^N}\left\|\frac{\varphi(\cdot-y)-\varphi(\cdot)}{|y|^s}\right\|_{L^2(\mathbb{R}^N)}^2\,\frac{dy}{|y|^N} \right)^\frac{1}{2}\,\left(\int_{\mathbb{R}^N} P_z(y)^2\,|y|^{N+2\,s}\,dy\right)^\frac{1}{2}\\
&=[\varphi]_{W^{s,2}(\mathbb{R}^N)}\,\left(\int_{\mathbb{R}^N} \frac{1}{z^{2\,N}}\,P_1\left(\frac{y}{z}\right)^2\,|y|^{N+2\,s}\,dy\right)^\frac{1}{2}.
\end{split}
\]
We now observe that
\[
\left(\int_{\mathbb{R}^N} \frac{1}{z^{2\,N}}\,P_1\left(\frac{y}{z}\right)^2\,|y|^{N+2\,s}\,dy\right)^\frac{1}{2}= z^s\,\left(\int_{\mathbb{R}^N} P_1(x)^2\,|x|^{N+2\,s}\,dx\right)^\frac{1}{2}\le z^s\,\sqrt{\beta_{N,s}},
\]
thus we get the conclusion.
\end{proof}

\begin{oss}
\label{oss:constants}
The constant $\beta_{N,s}$ (see e.g. \cite{FMM}) is given explicitly by
$$
\beta_{N,s}=\pi^{-\frac{N}{2}}\,\frac{\Gamma\left(\dfrac{N+2s}{2} \right)}{\Gamma(s)},$$
and therefore:
\begin{itemize}
\item $\beta_{N,s}$  is uniformly bounded for $s\nearrow 1$;
\vskip.2cm 
\item $\beta_{N,s}\sim s$, as $s \searrow 0$.
\end{itemize}
The value of the constant $\gamma_{N,s}$ can be found in \cite[Remark 3.11]{CS1}, and is given by
$$\gamma_{N,s}= \frac{2\,d_s}{C_{N,s}},$$
where
$$C_{N,s}=\pi^{-\frac{N}{2}}\,2^{2\,s}\,\frac{\Gamma\left(\dfrac{N+2s}{2} \right)}{\Gamma(2-s)}\, s\,(1-s)\qquad \mbox{and} \qquad d_s=2^{2\,s-1}\,\frac{\Gamma(s)}{\Gamma(1-s)}.$$
Hence, in particular, we have that: 
\begin{itemize}
\item $\gamma_{N,s}$ is uniformly bounded as $s\nearrow 1$;
\vskip.2cm
\item $\gamma_{N,s}\sim s^{-2}$, as $s\searrow 0$.
\end{itemize}
\end{oss}

\section{The fractional Faber-Krahn inequality by extension}
\label{sec:3}

As explained in the Introduction, we want to give a proof of the fractional Faber-Krahn inequality
\[
\lambda_{s,q}(\Omega)\ge \lambda_{s,q}(B),\qquad \mbox{ where } B \mbox{ is any ball such that } |B|=|\Omega|,
\]
by using symmetrization techniques in $\mathbb{R}^{N+1}_+$. The proof of Theorem \ref{teo:main} will be based on introducing quantitative elements in this proof. 
\par
The following expedient result will be useful. It asserts that in order to prove the fractional Faber-Krahn inequality (and its quantative version) for sets with finite measure, we can reduce to consider \textit{bounded} sets. The proof is quite easy, we leave it to the reader.
\begin{lm}[Reduction to bounded sets]\label{bounded-set}
Let $0<s<1$ and $1\le q<2^*_s$. For every open set $\Omega\subset\mathbb{R}^N$ with finite measure and every $R>0$, we define
\[
\Omega_R=\Omega\cap B_R(0).
\]
Then we have
\[
\lim_{R\to +\infty}\lambda_{s,q}(\Omega_R)=\lambda_{s,q}(\Omega)\qquad \mbox{ and }\qquad \lim_{R\to +\infty} \mathcal{A}(\Omega_R)=\mathcal{A}(\Omega).
\]
\end{lm}
As in \cite{FMM}, we define in $\mathbb{R}^{N+1}_+$ the {\it partial Schwarz symmetrization} $U^*$ of a nonnegative function $U\in \mathcal{H}^{1,s}(\mathbb{R}^{N+1}_+)$. By construction, the function $U^*$ is obtained by taking for almost every $z>0$, the $N-$dimensional Schwarz symmetrization of 
\[
x\mapsto U(x,z).
\]
More precisely: for almost every fixed $z>0$, the function $U^*(\cdot, z)$ is defined to be the unique radially symmetric decreasing function on $\mathbb R^N$ such that for all $t>0$
$$
|\{x\in\mathbb{R}^N\, :\, U^*(x, z)>t\}|=|\{x\in\mathbb{R}^N\, :\, U(x, z)>t\}|.
$$
\begin{prop}
Let $N\ge 2$ and $0<s<1$. Let $\varphi\in W^{s,2}(\mathbb{R}^N)\cap L^{(2^*_{1-s})'}(\mathbb{R}^N)$ be a nonnegative function. By taking $U_\varphi\in\mathcal{H}^{1,s}(\mathbb{R}^{N+1}_+)$ to be the minimizer of \eqref{min}, we have that 
\[
U^*_\varphi\in \mathcal{H}^{1,s}(\mathbb{R}^{N+1}_+),
\]
and the following P\'olya-Szeg\H{o} inequality holds
\begin{equation}\label{poliaszego}
\iint_{\mathbb{R}^{N+1}_+} z^{1-2\,s}\,|\nabla U_\varphi^*|^2\,dx\,dz\leq\iint_{\mathbb{R}^{N+1}_+} z^{1-2\,s}\,|\nabla U_\varphi|^2\,dx\,dz.
\end{equation}
Moreover, we have $\mathrm{trace}(U_\varphi^*)=\varphi^*$.
In particular, we get
\begin{equation}
\label{energia}
\gamma_{N,s}\,\iint_{\mathbb{R}^{N+1}_+} z^{1-2\,s}\,|\nabla U_\varphi^*|^2\,dx\,dz\ge [\varphi^*]_{W^{s,2}(\mathbb{R}^N)}^2.
\end{equation}
\end{prop}
\begin{proof}
For the first statement, we follow the ideas contained in \cite[Lemma 2.6]{FMM}. First, it is easy to see that
\[
\iint_{\mathbb{R}^{N+1}_+} z^{1-2\,s}\,|U_\varphi^*|^2\,dx\,dz=\iint_{\mathbb{R}^{N+1}_+} z^{1-2\,s}\,|U_\varphi|^2\,dx\,dz,
\]
since $U_\varphi(\cdot,z)$ and $U^*_\varphi(\cdot,z)$ are equi-measurable.
Moreover, by the classical P\'olya-Szeg\H{o} inequality for the Schwarz symmetrization of $x\mapsto U(x,z)$ in $\mathbb{R}^N$ for a.e. $z>0$ fixed, one has also
\[
\iint_{\mathbb{R}^{N+1}_+} z^{1-2\,s}\,|\nabla_x U_\varphi^*|^2\,dx\,dz\leq\iint_{\mathbb{R}^{N+1}_+} z^{1-2\,s}\,|\nabla_x U_\varphi|^2\,dx\,dz.
\]
We only have to prove that
\begin{equation}\label{polyaz}
\iint_{\mathbb{R}^{N+1}_+} z^{1-2\,s}\,\left|\partial_z U_\varphi^*\right|^2\,dx\,dz\leq\iint_{\mathbb{R}^{N+1}_+} z^{1-2\,s}\,\left|\partial_z U_\varphi\right|^2\,dx\,dz.
\end{equation}
Since our space $\mathcal{H}^{1,s}(\mathbb{R}^{N+1}_+)$ is contained in the functional space used in \cite{FMM}, then the result for $0<s<1/2$ follows immediately from \cite[Lemma 2.6]{FMM}. 
\par
For the case $1/2\le s<1$ some care is needed, due to the singularity of our weight. In this case, we define the regularized weight  
$$
\rho_\varepsilon(z)=(\varepsilon^2+z^2)^\frac{1-2\,s}{2},\qquad\varepsilon>0,\, z\in \mathbb{R},
$$
and set
\[
\widetilde{U}_\varphi(x,z)=\left\{\begin{array}{rl}
U_\varphi(x,z),\qquad \mbox{ for } z\ge 0,\\
U_\varphi(x,-z),\qquad \mbox{ for } z<0.
\end{array}
\right.
\]
By construction, we have
\[
\iint_{\mathbb{R}^{N+1}_+} z^{1-2\,s}\,|\partial_z U_\varphi|^2\,dz\ge \iint_{\mathbb{R}^{N+1}_+} \rho_\varepsilon(z)\,|\partial_z U_\varphi|^2\,dz=\frac{1}{2}\,\iint_{\mathbb{R}^{N+1}} \rho_\varepsilon(z)\,|\partial_z \widetilde U_\varphi|^2\,dz.
\]
We can now reproduce step by step the proof of \cite[Lemma 2.6]{FMM}. This is based on an iterative use of Steiner symmetrizations in the $x-$space, in conjuction with the weighted P\'olya-Szeg\H{o} inequality of \cite[Theorem 1]{BRO}. This shows that
\[
\iint_{\mathbb{R}^{N+1}} \rho_\varepsilon(z)\,\left|\partial_z (\widetilde U_\varphi)^*\right|^2\,dx\,dz\leq\iint_{\mathbb{R}^{N+1}} \rho_\varepsilon(z)\,|\partial_z \widetilde U_\varphi|^2\,dz.
\]
It is only left to observe that $(\widetilde U_\varphi)^*$ is still even in the $z-$variable, thus we obtain
\[
\iint_{\mathbb{R}^{N+1}_+} z^{1-2\,s}\,|\partial_z U_\varphi|^2\,dz\ge \iint_{\mathbb{R}^{N+1}_+} \rho_\varepsilon(z)\,|\partial_z U_\varphi^*|^2\,dz.
\]
By taking the limit as $\varepsilon$ goes to $0$ and using Fatou's Lemma, we get \eqref{polyaz}.
\par
Once we obtained that $U^*_\varphi$ is in the relevant Sobolev space, the fact that
\[
\mathrm{trace\,}(U^*_\varphi)\in W^{s,2}(\mathbb{R}^N),
\]
follows from Lemma \ref{prop:space}.
\par
In order to identify the trace of $U_\varphi^*$, we use the estimate \eqref{lemma2} and the fact that the Schwarz symmetrization is non-expansive in $L^2(\mathbb{R}^N)$ (see \cite[Theorem 3.5]{LL}). This entails
\[
\|U_\varphi^*(\cdot,z)-\varphi^*\|^2_{L^2(\mathbb{R}^N)}\le \|U_\varphi(\cdot,z)-\varphi\|^2_{L^2(\mathbb{R}^N)}\le \beta_{N,s}\,[\varphi]^2_{W^{s,2}(\mathbb{R}^N)}\,z^{2\,s},\qquad \mbox{ for } z>0,
\]
which shows that
\[
U^*_\varphi (\cdot,z)\to \varphi^*\quad \mbox{ in } L^2(\mathbb{R}^N), \mbox{ as } z\to 0^+.
\] 
This permits to conclude that $\mathrm{trace}(U^*_\varphi)=\varphi^*$. 
\par
Finally, let us prove \eqref{energia}. We take $U_{\varphi^*}\in \mathcal{H}^{1,s}(\mathbb{R}^{N+1}_+)$ to be the minimizer of \eqref{min} with boundary datum $\varphi^*$. It is now sufficient to use the minimality of $U_{\varphi^*}$ and the fact that
\[
\mathrm{trace}(U^*_\varphi)=\mathrm{trace}(U_{\varphi^*})=\varphi^*,
\]
to get that
\[
\iint_{\mathbb{R}^{N+1}_+} z^{1-2\,s}\,|\nabla U_\varphi^*|^2\,dx\,dz\ge \iint_{\mathbb{R}^{N+1}_+} z^{1-2\,s}\,|\nabla U_{\varphi^*}|^2\,dx\,dz.
\]
By recalling \eqref{equiv}, we eventually get the conclusion. 
\end{proof}

\begin{teo}[Faber-Krahn inequality]
\label{teo:FK}
Let $N\ge 2$, $0<s<1$ and $1\le q<2^*_s$. For every $\Omega\subset\mathbb{R}^N$ open set with finite measure, we have
\[
|\Omega|^{\frac{2}{q}-1+\frac{2\,s}{N}}\,\lambda_{s,q}(\Omega)\ge |B|^{\frac{2}{q}-1+\frac{2\,s}{N}}\,\lambda_{s,q}(B),
\]
where $B\subset\mathbb{R}^N$ is any $N-$dimensional ball.
\end{teo}
\begin{proof}
By Lemma \ref{bounded-set}, it is enough to prove the result under the further assumption that $\Omega$ is bounded.
By scale invariance of the Faber-Krahn inequality, we can assume without loss of generality that $|\Omega|=1$. We now take $u_\Omega\in\mathcal{D}^{s,2}_0(\Omega)$ positive such that
\[
\|u_\Omega\|_{L^q(\Omega)}=1\qquad \mbox{ and }\qquad [u_\Omega]_{W^{s,2}(\mathbb{R}^N)}^2=\lambda_{s,q}(\Omega). 
\]
For ease of notation, we denote by $U_\Omega$ the extension of $u_\Omega$, obtained by the convolution in \eqref{poiss}. Observe that for an open bounded set $\Omega\subset\mathbb{R}^N$, we have
\[
\mathcal{D}^{s,2}_0(\Omega)\subset W^{s,2}(\mathbb{R}^N)\cap L^{(2^*_{1-s})'}(\mathbb{R}^N),
\]
thus by Proposition \ref{mini}, we know that $U_\Omega\in \mathcal{H}^{1,s}(\mathbb{R}^{N+1}_+)$. Moreover, recalling \eqref{equiv} and using the generalized P\'olya-Szeg\H{o} principle \eqref{poliaszego}, we get
\begin{equation}
\label{PS-FMM}
\begin{split}
\lambda_{s,q}(\Omega)=[u_\Omega]_{W^{s,2}(\mathbb{R}^N)}^2&=\gamma_{N,s}\,\iint_{\mathbb{R}^{N+1}_+} z^{1-2\,s}\,|\nabla U_\Omega|^2\,dx\,dz\ge \gamma_{N,s}\,\iint_{\mathbb{R}^{N+1}_+} z^{1-2\,s}\,|\nabla U_\Omega^*|^2\,dx\,dz.
\end{split}
\end{equation}
By further using \eqref{energia}, we get
\[
\lambda_{s,q}(\Omega)\ge [u^*_\Omega]_{W^{s,2}(\mathbb{R}^N)}^2.
\]
Since $u^*\in W^{s,2}(\mathbb{R}^N)$ and $u^*_\Omega=0$ almost everywhere in $\mathbb{R}^N\setminus B_\Omega$, by \cite[Proposition B.1]{BPS} we get that $u^*_\Omega\in \mathcal{D}^{s,2}_0(B_\Omega)$. We recall that $B_\Omega$ denotes the ball centered at the origin and such that $|B_\Omega|=|\Omega|$.
\par
Moreover, by construction,
\[
\|u^*_\Omega\|_{L^q(B_\Omega)}=\|u_\Omega\|_{L^q(\Omega)}=1.
\]
Thus we get
\[
\lambda_{s,q}(\Omega)\ge [u^*_\Omega]_{W^{s,2}(\mathbb{R}^N)}^2\ge \lambda_{s,q}(B_\Omega),
\]
as desired. 
\end{proof}

\section{Estimates on level sets}
\label{sec:4}

\subsection{An expedient estimate}

The following technical result will be useful in order to transfer the asymmetry from a set to another. 
This is a generalization of \cite[Lemma 2.8]{BD}.
\begin{lm}[Transfer of asymmetry]
\label{lm:asimmetria}
Let $\Omega, E\subset\mathbb{R}^N$ be two measurable sets with finite measure, such that
\[
\frac{|\Omega\Delta E|}{|\Omega|}\le \gamma\,\mathcal{A}(\Omega),
\]
for some $0<\gamma<1/2$.
Then
\[
\mathcal{A}(E)\ge \frac{1-2\,\gamma}{c_\gamma}\,\mathcal{A}(\Omega),\qquad \mbox{ where } c_\gamma=\left\{\begin{array}{rl}
1,& \mbox{ if } |E\setminus \Omega|=0,\\
1+2\,\gamma,& \mbox{ if } |E\setminus \Omega|>0. 
\end{array}
\right.
\]
\end{lm}
\begin{proof} 
We can suppose that $\mathcal{A}(\Omega)>0$, otherwise there is nothing to prove.
We take a ball $B$ such that $|B|=|E|$ and
\[
\mathcal{A}(E)=\frac{|E\Delta B|}{|E|},
\]
and call $B'$ the ball concentric with $B$, such that $|B'|=|\Omega|$.
We recall that 
\[
|\Omega\Delta E|=\|1_\Omega-1_E\|_{L^1(\mathbb{R}^N)},
\]
thus by using the triangle inequality, we get
\[
\begin{split}
\mathcal{A}(E)=\frac{|E\Delta B|}{|E|}&\ge \frac{|\Omega|}{|E|}\,\left(\frac{|\Omega\Delta B'|}{|\Omega|}-\frac{|B'\Delta B|}{|\Omega|}-\frac{|\Omega\Delta E|}{|\Omega|}\right)\\
&\ge \frac{|\Omega|}{|E|}\,\left(\mathcal{A}(\Omega)-2\,\frac{|\Omega\Delta E|}{|\Omega|}\right)\\
&\ge \frac{|\Omega|}{|E|}\,(1-2\,\gamma)\,\mathcal{A}(\Omega).
\end{split}
\]
Observe that in the second inequality we used that
\[
|B'\Delta B|=\Big||\Omega|-|E|\Big|\le |\Omega\Delta E|,
\]
while in the third one we used the hypothesis.
In order to conclude, we only need to bound from below the ratio $|\Omega|/|E|$. If $|E\setminus \Omega|=0$, then we get 
\[
\frac{|\Omega|}{|E|}=\frac{|\Omega|}{|E\setminus\Omega|+|E\cap \Omega|}\ge 1.
\]
If $|E\setminus \Omega|>0$, we observe that
\[
\frac{|\Omega|}{|E|}=\frac{|\Omega|}{|E\setminus\Omega|+|\Omega\cap E|}\ge \frac{|\Omega|}{|\Omega\Delta E|+|\Omega|}\ge \frac{1}{1+\gamma\,\mathcal{A}(\Omega)}.
\]
By recalling that the Fraenkel asymmetry is always smaller than $2$, we get the desired conclusion.
\end{proof}
\subsection{Closeness of level sets}
For an open bounded set $\Omega\subset\mathbb{R}^N$ such that $\mathcal{A}(\Omega)>0$, throughout this section we fix $u_\Omega$ to be the optimal function for $\lambda_{s,q}(\Omega)$ defined in Lemma \ref{lm:eigenfunction}. As in the previous section, we set 
\[
U_\Omega(x,z)=P_z\ast u_\Omega(x),\qquad (x,z)\in\mathbb{R}^{N+1}_+.
\]
Then we define
\begin{equation}
\label{livellone}
T:=\sup\left\{ t>0\, :\, |\{x\in\Omega\, :\, u_\Omega(x)>t\}|\ge |\Omega|\,\left(1-\frac{1}{9}\,\mathcal{A}(\Omega)\right)\right\}.
\end{equation}
Observe that we have $T>0$, since the function
\[
t\mapsto |\{x\in\Omega\, :\, u_\Omega(x)>t\}|,
\]
is non-increasing right-continuous and 
\[
|\{x\in\Omega\, :\, u_\Omega(x)>0\}|=|\Omega|.
\]
Moreover, still thanks to the right-continuity, it is easy to see that
\begin{equation}
\label{misuraT}
|\{x\in\Omega\, :\, u_\Omega(x)>T\}|\le |\Omega|\,\left(1-\frac{1}{9}\,\mathcal{A}(\Omega)\right).
\end{equation}
\begin{lm}
\label{lm:stimaL1}
We fix $T>0$ as in \eqref{livellone} and $\alpha>0$, then
\[
\mbox{ for every } \frac{T}{4}\le t\le \frac{3}{8}\, T \qquad \mbox{ and }\qquad  \mbox{ for every }0<z\le \left(\frac{T}{8\,\sqrt{\alpha\,\beta_{N,s}\,\lambda_{s,q}(\Omega)}}\right)^\frac{1}{s},
\]
we have
\[
\Big|\left\{x\in\Omega\, :\, u_\Omega(x)>\frac{T}{2}\right\}\setminus \left\{x\in\mathbb{R}^N\, :\, U_\Omega(x,z)>t\right\}\Big|\le \frac{1}{\alpha},
\]
and
\[
\left|\left\{x\in\mathbb{R}^N\, :\, U_\Omega(x,z)>t\right\}\setminus\left\{x\in\Omega\, :\, u_\Omega(x)>\frac{T}{8}\right\}\right|\le \frac{1}{\alpha}.
\]
\end{lm}
\begin{proof}
By \eqref{lemma2}, we have 
\[
\|U_\Omega(\cdot,z)-u_\Omega\|^2_{L^2(\mathbb{R}^N)}\le \beta_{N,s}\,\lambda_{s,q}(\Omega)\,z^{2\,s},\qquad \mbox{ for a.\,e. }z>0,
\]
where we also used the minimality of $u_\Omega$.
Then by using Markov-Chebychev's inequality, we get for $z>0$
\begin{equation}
\label{misura}
\left|\{x\in\mathbb{R}^N\, :\, |U_\Omega(x,z)-u_\Omega(x)|>\sqrt{\beta_{N,s}\,\lambda_{s,q}(\Omega)\,\alpha}\,z^s \}\right|\le \frac{1}{\alpha}.
\end{equation}
We now take $t$ and $z$ as in the statement, then for every $x$ such that $u_\Omega(x)>T/2$ and $U_\Omega(x,z)\le t$, we have
\[
u_\Omega(x)-U_\Omega(x,z)>\frac{T}{2}-t\ge \frac{T}{8}\ge \sqrt{\beta_{N,s}\,\lambda_{s,q}(\Omega)\,\alpha}\,z^s,
\]
that is
\[
\begin{split}
\Big(\left\{x\in\Omega\, :\,u_\Omega(x)>\frac{T}{2}\right\}&\setminus \left\{x\in\mathbb{R}^N\, :\, U_\Omega(x,z)>t\right\}\Big)\\
&\subset \left\{x\in\mathbb{R}^N\, :\, |U_\Omega(x,z)-u_\Omega(x)|>\sqrt{\beta_{N,s}\,\lambda_{s,q}(\Omega)\,\alpha}\,z^s\right \}.
\end{split}
\]
By using \eqref{misura}, we get 
\[
\left|\left\{x\in\Omega\, :\, u_\Omega(x)>\frac{T}{2}\right\}\setminus \left\{x\in\mathbb{R}^N\, :\, U_\Omega(x,z)>t\right\}\right|\le \frac{1}{\alpha},
\]
as desired. The second estimate is proved in a similar way, we leave the details to the reader.
\end{proof}

\begin{prop}
\label{prop:asimmetrie}
We fix $T>0$ as in \eqref{livellone}, then
\[
\mbox{ for } \frac{T}{4}\le t\le \frac{3}{8}\, T \qquad \mbox{ and }\qquad \mbox{ for }0<z\le \left(\frac{\sqrt{\mathcal{A}(\Omega)\,|\Omega|}}{24\,\sqrt{\beta_{N,s}\,\lambda_{s,q}(\Omega)}}\,T\right)^\frac{1}{s},
\]
we have
\begin{equation}
\label{eddai!}
\Big|\big|\left\{x\in\mathbb{R}^N\, :\, U_\Omega(x,z)>t\right\}\big|-|\Omega|\Big|\le \frac{1}{3}\,|\Omega|\,\mathcal{A}(\Omega),
\end{equation}
and
\begin{equation}
\label{asimmetria}
\mathcal{A}\big(\left\{x\in\mathbb{R}^N\, :\, U_\Omega(x,z)>t\right\}\big)\ge \frac{1}{5}\,\mathcal{A}(\Omega).
\end{equation}
\end{prop}
\begin{proof}
For ease of notation, we set
\[
\Omega_t=\left\{x\in\Omega\, :\, u_\Omega(x)>t\right\},
\]
and
\[
E_{t,z}=\left\{x\in\mathbb{R}^N\, :\, U_\Omega(x,z)>t\right\},\qquad \mbox{ for } t\in \left[\frac{T}{4},\frac{3}{8}\,T\right]\, \mbox{ and }\, z\in \left(0,\left(\frac{\sqrt{\mathcal{A}(\Omega)\,|\Omega|}}{24\,\sqrt{\beta_{N,s}\,\lambda_{s,q}(\Omega)}}\,T\right)^\frac{1}{s}\right].
\]
By definition \eqref{livellone} of the level $T$, we have that
\begin{equation}
\label{ciao}
\frac{|\Omega\setminus \Omega_{T/2}|}{|\Omega|}\le \frac{1}{9}\,\mathcal{A}(\Omega).
\end{equation}
We now observe that by using \eqref{ciao} and Lemma \ref{lm:stimaL1} with the choice 
\[
\alpha=\frac{9}{\mathcal{A}(\Omega)\,|\Omega|},
\] 
we get
\[
\begin{split}
\frac{|E_{t,z}\Delta \Omega|}{|\Omega|}&=\frac{|\Omega\setminus E_{t,z}|}{|\Omega|}+\frac{|E_{t,z}\setminus \Omega|}{|\Omega|}\\
&\le \frac{|\Omega\setminus \Omega_{T/2}|}{|\Omega|}+\frac{|\Omega_{T/2}\setminus E_{t,z}|}{|\Omega|}+\frac{|E_{t,z}\setminus \Omega_{T/8}|}{|\Omega|}\\
&\le \frac{|\Omega\setminus \Omega_{T/2}|}{|\Omega|}+\frac{2}{\alpha}\,\frac{1}{|\Omega|}\le \frac{1}{3}\,\mathcal{A}(\Omega).
\end{split}
\]
Finally, by triangle inequality we have
\[
|\Omega|-|E_{t,z}\Delta \Omega|\le |E_{t,z}|\le |\Omega|+|E_{t,z}\Delta \Omega|,
\]
thus by joining the last two estimates we get \eqref{eddai!}.
\par
We can now apply Lemma \ref{lm:asimmetria} with $\gamma=1/3$, so to obtain
\[
\mathcal{A}(E_{t,z})\ge \frac{1-\dfrac{2}{3}}{1+\dfrac{2}{3}}\,\mathcal{A}(\Omega)=\frac{1}{5}\,\mathcal{A}(\Omega),
\]
which proves \eqref{asimmetria}.
\end{proof}
\subsection{A remainder term}
We now introduce some quantitative elements in the proof of the Faber-Krahn inequality presented in Theorem \ref{teo:FK}. With this aim, we need to recall the {\it sharp quantitative isoperimetric inequality}
\begin{equation}
\label{isoquant}
|\Omega|^{\frac{1-N}{N}}\,P(\Omega)-|B|^{\frac{1-N}{N}}\,P(B)\ge \Theta_N\,\mathcal{A}(\Omega)^2,
\end{equation}
proved in \cite[Theorem 1.1]{FMP} (see also \cite[Theorem 4.3]{CL}). Here $P$ denotes the distributional perimeter of a set. A possible explicit value for the constant $\Theta_N$ is computed in \cite[equation (1.12)]{FigMP}. In our notation, this reads
\[
\Theta_N=\omega_N^{1/N}\,\frac{\left(2-2^\frac{N-1}{N}\right)^3}{(181)^2\,N^{13}}.
\]
\begin{prop}[An enhanced P\'olya-Szeg\H{o}--type estimate]
\label{prop:PS}
Let $0<s<1$ and let $\Omega\subset\mathbb{R}^N$ be an open bounded set. For $t>0$ and $z>0$, we set
\[
E_{t,z}=\left\{x\in\mathbb{R}^N\, :\, U_\Omega(x,z)>t\right\}\qquad \mbox{ and }\qquad \mu_z(t)=|E_{t,z}|.
\]
Then for every ball $B\subset\mathbb{R}^N$ such that $|B|=|\Omega|$, 
we have
\[
\begin{split}
\lambda_{s,q}(\Omega)&-\lambda_{s,q}(B)\ge C_1\,\int_0^{+\infty} z^{1-2\,s}\left(\int_0^{+\infty} \mathcal{A}(E_{t,z})^2\,\frac{\left(\mu_{z}(t)^\frac{N-1}{N}\right)^2}{-\mu'_z(t)}\,dt\right)\,dz,
\end{split}
\]
where the constant $C_1=C_1(N,s)>0$ is given by
\begin{equation}
\label{C1}
C_1=2\,N\,\omega_N^{1/N}\,\Theta_N\,\gamma_{N,s}>0,
\end{equation}
and $\gamma_{N,s}$ is the same as in \eqref{equiv}.
\end{prop}
\begin{proof}
We introduce some quantitative informations into the generalized P\'olya-Szeg\H{o} principle used in \eqref{PS-FMM}. We have seen that
\begin{equation}
\label{conto0}
\lambda_{s,q}(\Omega)=[u_\Omega]_{W^{s,2}(\mathbb{R}^N)}^2
=\gamma_{N,s}\,\iint_{\mathbb{R}^{N+1}_+} z^{1-2\,s}\,|\nabla_x U_\Omega|^2\,dx\,dz
+\gamma_{N,s}\,\iint_{\mathbb{R}^{N+1}_+} z^{1-2\,s}\,\left|\partial_z U_\Omega\right|^2\,dx\,dz.
\end{equation}
For the $z-$derivative, we already observed that
\[
\iint_{\mathbb{R}^{N+1}_+} z^{1-2\,s}\,\left|\partial_z U_\Omega\right|^2\,dx\,dz\ge \iint_{\mathbb{R}^{N+1}_+} z^{1-2\,s}\,\left|\partial_z U^*_{\Omega}\right|^2\,dx\,dz.
\]
For the $x-$derivative, we proceed as in the local case: by using the coarea formula, this can be written as
\begin{equation}
\begin{split}
\label{conto1}
&\iint_{\mathbb{R}^{N+1}_+} z^{1-2\,s}\,|\nabla_x U_\Omega|^2\,dx\,dz\\
&\hspace{2em}=\int_{0}^{+\infty}z^{1-2s}\left(\int_{0}^{+\infty}\left(\int_{\{x\in\mathbb{R}^N \, : \, U_\Omega(x,z)=t\}} \,|\nabla_x U_\Omega|^2\,\frac{d\mathcal{H}^{N-1}(x)}{|\nabla_x U_\Omega|}\right)\,dt\right)\,dz\\
&\hspace{2em}\ge \int_{0}^{+\infty}z^{1-2s}\left(\int_{0}^{+\infty}\frac{P(E_{t,z})^2}{\displaystyle\int_{\{x\in\mathbb{R}^N \, : \, U_\Omega(x,z)=t\}}\frac{d\mathcal{H}^{N-1}(x)}{|\nabla_x U_\Omega|}}\,dt\right)\,dz
\end{split}
\end{equation}
where $P(E_{t,z})$ denotes the perimeter of the set $E_{t,z}$, and we have used Jensen's inequality. 
Following the same computation as in \cite[Lemma 2.9]{BD}, defining
$$
E_{t,z}^*=\{x\in\mathbb{R}^N \ : \ U_\Omega^*(x,z)>t\},
$$
and using  the quantitative isoperimetric inequality \eqref{isoquant}, one can prove that
\[
P(E_{t,z})^2\geq P(E_{t,z}^*)^2+2\,N\,\omega_N^{1/N}\,\Theta_N\,\mu_z(t)^{2\,\frac{N-1}{N}}\mathcal{A}(E_{t,z})^2.
\]
Thus we obtain 
\[
\begin{split}
\iint_{\mathbb{R}^{N+1}_+} z^{1-2\,s}\,|\nabla_x U_\Omega|^2\,dx\,dz
&\ge \int_{0}^{+\infty}z^{1-2s}\left(\int_{0}^{+\infty}\frac{P(E_{t,z}^*)^2}{-\mu_z'(t)}\,dt\right)\,dz\\
&+2\,N\,\omega_N^{1/N}\,\Theta_N\,\int_{0}^{+\infty} z^{1-2\,s}\,\left(\int_0^{+\infty}\frac{\left(\mu_z(t)^\frac{N-1}{N}\right)^2\,\mathcal{A}(E_{t,z})^2}{-\mu'_z(t)} dt\right)\,dz.
\end{split}
\]
Observe that in the right-hand side, we also used the fact that
\begin{equation}
\label{distribution}
-\mu_z'(t)\geq\int_{\{x\in\mathbb{R}^N \ : \ U_\Omega(x,z)=t\}}\frac{d\mathcal{H}^{N-1}(x)}{|\nabla_x U_\Omega|}.
\end{equation}
We remark that one has equality in \eqref{conto1} for a radial function, since the modulus of the gradient is constant on each level set. Moreover, the isoperimetric inequality is obviously an equality in the case of balls. Finally, for the symmetrized function $U^*_\Omega$, one has the equality also in \eqref{distribution} (see \cite[Proposition 2.4]{FMM}). 
\par
By using these facts, we can conclude that
$$\int_{0}^{+\infty}z^{1-2s}\left(\int_{0}^{+\infty}\frac{P(E_{t,z}^*)^2}{-\mu_z'(t)}\,dt\right)\,dz=\iint_{\mathbb{R}^{N+1}_+} z^{1-2\,s}\,|\nabla_x U^*_\Omega|^2\,dx\,dz.$$
Moreover, we have seen in the proof of Theorem \ref{teo:FK} that
\[
\gamma_{N,s}\,\iint_{\mathbb{R}^{N+1}_+} z^{1-2\,s}\,|\nabla U^*_\Omega|^2\,dx\,dz\geq\lambda_{s,q}(B).
\]
Hence, coming back to \eqref{conto0}, we have
\[
\begin{split}
\lambda_{s,q}(\Omega)&=\gamma_{N,s}\,\iint_{\mathbb{R}^{N+1}_+} z^{1-2\,s}\,|\nabla_x U_\Omega|^2\,dx\,dz
+\gamma_{N,s}\,\iint_{\mathbb{R}^{N+1}_+} z^{1-2\,s}\,\left|\partial_z U_\Omega\right|^2\,dx\,dz\\
&\geq\lambda_{s,q}(B)+2\,\gamma_{N,s}N\,\omega_N^{1/N}\,\Theta_N\,\int_{0}^{+\infty} z^{1-2\,s}\,\left(\int_0^{+\infty}\frac{\left(\mu_z(t)^\frac{N-1}{N}\right)^2\,\mathcal{A}(E_{t,z})^2}{-\mu'_z(t)} dt\right)\,dz.
\end{split}
\]
This concludes the proof.
\end{proof}

\section{Proof of the main result}
\label{sec:5}

\subsection{Proof of Theorem \ref{teo:main}}

Thanks to the scale invariance, we can assume that $|\Omega|=1$. By Lemma \ref{bounded-set}, we can further assume $\Omega$ to be bounded. Thus we have to prove that
\begin{equation}
\label{voglio}
\lambda_{s,q}(\Omega)-\lambda_{s,q}(B)\ge C\,\mathcal{A}(\Omega)^\frac{3}{s},
\end{equation}
where $B$ is any ball such that $|B|=|\Omega|=1$. 
We also observe that if $\lambda_{s,q}(\Omega)>2\,\lambda_{s,q}(B)$, then by using that $\mathcal{A}(\Omega)<2$
\[
\lambda_{s,q}(\Omega)-\lambda_{s,q}(B)>\lambda_{s,q}(B)=\frac{\lambda_{s,q}(B)}{2^\frac{3}{s}}\,2^\frac{3}{s}>\left(\frac{\lambda_{s,q}(B)}{2^\frac{3}{s}}\right)\,\mathcal{A}(\Omega)^\frac{3}{s},
\]
i.e. we get the desired estimate \eqref{voglio}, with
\[
C=\frac{\lambda_{s,q}(B)}{2^\frac{3}{s}}.
\]
Thus, we can confine ourselves to consider the case 
\begin{equation}
\label{restrizione}
\lambda_{s,q}(\Omega)\le 2\,\lambda_{s,q}(B).
\end{equation}
We now set
\begin{equation}
\label{C2}
C_2=\frac{1}{9}\,\left(\frac{2}{q}+\frac{2\,s}{N}-1\right).
\end{equation}
Observe that $C_2>0$, thanks to the fact that $q<2^*_s$.
We define
\[
T_0=\frac{C_2}{4\,(1+C_2)}\,\mathcal{A}(\Omega),
\]
then we have two possibilities for the value $T$ defined in \eqref{livellone}:
\[
\mbox{either }\qquad T\le T_0\qquad \mbox{ or }\qquad T>T_0.
\]
{\bf Case $T\le T_0$.} This is the easy case, here we do not need to work with the extension in $\mathbb{R}^{N+1}_+$. In particular, Proposition \ref{prop:PS} is not needed here.
\par
We consider the set $\Omega_T=\{x\in\Omega\, :\, u_\Omega(x)>T\}$, which is open thanks to the continuity of $u_\Omega$. We also verify that this set is not empty. Indeed, by using Minkowski's inequality and the fact that
\[
u_\Omega\le (u_\Omega-T)_++T,
\] 
we have (recall that $|\Omega|=1$ and $\int_\Omega u_\Omega^q\,dx=1$)
\begin{equation}
\label{nonvuoto}
\begin{split}
\left(\int_{\Omega_T} (u_\Omega-T)^q_+\,dx\right)^\frac{1}{q}&= \left(\int_{\Omega} (u_\Omega-T)_+^q\,dx\right)^\frac{1}{q}\\
&\ge \left(\int_{\Omega} u^q_\Omega\,dx\right)^\frac{1}{q}-T\,|\Omega|^\frac{1}{q}=1-T.
\end{split}
\end{equation}
By observing that $T\le T_0<1/2$, the last estimate ensures that $\Omega_T$ has positive measure.
\par
We now use the function $(u_\Omega-T)_+$ in the variational definition of $\lambda_{s,q}(\Omega_T)$, we get
\[
\lambda_{s,q}(\Omega_T)\le \frac{\Big[(u_\Omega-T)_+\Big]^2_{W^{s,2}(\mathbb{R}^N)}}{\displaystyle\left(\int_{\Omega_T} (u_\Omega-T)^q_+\,dx\right)^\frac{2}{q}}.
\]
By using that
\[
\Big[(u_\Omega-T)_+\Big]^2_{W^{s,2}(\mathbb{R}^N)}\le [u_\Omega]^2_{W^{s,2}(\mathbb{R}^N)}=\lambda_{s,q}(\Omega),
\]
we then obtain
\[
\begin{split}
\lambda_{s,q}(\Omega)&\ge \lambda_{s,q}(\Omega_T)\,\left(\int_{\Omega_T} (u_\Omega-T)^q_+\,dx\right)^\frac{2}{q}\\
&\ge \lambda_{s,q}(B)\,|\Omega_T|^{1-\frac{2}{q}-\frac{2\,s}{N}}\,\left(\int_{\Omega_T} (u_\Omega-T)^q_+\,dx\right)^\frac{2}{q}.
\end{split}
\]
In the second inequality we used the Faber-Krahn inequality for $\lambda_{s,q}$, applied to the open set $\Omega_T$. By using \eqref{misuraT} and basic calculus\footnote{We use the convexity of the function $t\mapsto (1+t)^{-\alpha}$, for $\alpha>0$.}, we get
\[
|\Omega_T|^{1-\frac{2}{q}-\frac{2\,s}{N}}\ge \left(1-\frac{1}{9}\,\mathcal{A}(\Omega)\right)^{1-\frac{2}{q}-\frac{2\,s}{N}}\ge 1+\frac{1}{9}\,\left(\frac{2}{q}+\frac{2\,s}{N}-1\right)\,\mathcal{A}(\Omega).
\]
By recalling the definition of $C_2$, up to now we have obtained
\begin{equation}
\label{intermedio}
\lambda_{s,q}(\Omega)\ge \lambda_{s,q}(B)\,\left(1+C_2\,\mathcal{A}(\Omega)\right)\,\left(\int_{\Omega_T} (u_\Omega-T)^q_+\,dx\right)^\frac{2}{q},
\end{equation}
We now estimate the $L^q$ norm of $(u_\Omega-T)_+$: by raising to the power $2$ the estimate \eqref{nonvuoto} and observing that $T\le T_0<1/2$, we get
\[
\left(\int_{\Omega} (u_\Omega-T)_+^q\,dx\right)^\frac{2}{q}\ge (1-T)^2\ge 1-2\,T_0\ge 1-\frac{C_2}{2\,(1+C_2)}\,\mathcal{A}(\Omega).
\]
We insert this estimate in \eqref{intermedio}, so to obtain
\[
\lambda_{s,q}(\Omega)\ge \lambda_{s,q}(B)\,\left(1+C_2\,\mathcal{A}(\Omega)\right)\,\left(1-\frac{C_2}{2\,(1+C_2)}\,\mathcal{A}(\Omega)\right).
\]
The right-hand side can be estimated as follows
\[
\begin{split}
\left(1+C_2\,\mathcal{A}(\Omega)\right)&\,\left(1-\frac{C_2}{2\,(1+C_2)}\,\mathcal{A}(\Omega)\right)\\
&=1+\left[C_2-\frac{C_2}{2\,(1+C_2)}\right]\,\mathcal{A}(\Omega)-\frac{C_2^2}{2\,(1+C_2)}\,\mathcal{A}(\Omega)^2\\
&\ge 1+\left[C_2-\frac{C_2}{2\,(1+C_2)}\,-\frac{C_2^2}{(1+C_2)}\right]\,\mathcal{A}(\Omega)\\
&=1+\frac{C_2}{2\,(1+C_2)}\,\mathcal{A}(\Omega),
\end{split}
\]
thus we eventually get
\[
\lambda_{s,q}(\Omega)-\lambda_{s,q}(B)\ge \frac{C_2\,\lambda_{s,q}(B)}{2\,(1+C_2)}\,\mathcal{A}(\Omega)\ge \left(\frac{C_2\,\lambda_{s,q}(B)}{2^\frac{3}{s}\,(1+C_2)}\right)\,\mathcal{A}(\Omega)^\frac{3}{s},
\]
as desired.
\vskip.2cm\noindent
{\bf Case $T>T_0$.} If we set for simplicity
\begin{equation}
\label{z_0}
z_0=\left(\frac{\sqrt{\mathcal{A}(\Omega)\,|\Omega|}}{24\,\sqrt{2\,\beta_{N,s}\,\lambda_{s,q}(B)}}\,T\right)^\frac{1}{s},
\end{equation}
by assumption \eqref{restrizione}, we have
\begin{equation}
\label{z_01}
z_0< \left(\frac{\sqrt{\mathcal{A}(\Omega)\,|\Omega|}}{24\,\sqrt{\beta_{N,s}\,\lambda_{s,q}(\Omega)}}\,T\right)^\frac{1}{s}.
\end{equation}
We now want to use the enhanced P\'olya-Szeg\H{o}--type estimate of Proposition \ref{prop:PS}, in conjunction with Proposition \ref{prop:asimmetrie}.
Thus, we have
\[
\begin{split}
\lambda_{s,q}(\Omega)-\lambda_{s,q}(B)&\ge C_1\,\int_0^{+\infty} z^{1-2\,s}\,\left(\int_0^{+\infty} \mathcal{A}(E_{t,z})^2\,\frac{\left(\mu_{z}(t)^\frac{N-1}{N}\right)^2}{-\mu'_z(t)}\,dt\right)\,dz\\
&\ge C_1\,\int_0^{z_0} z^{1-2\,s}\,\left(\int_\frac{T}{4}^{\frac{3}{8}\,T} \mathcal{A}(E_{t,z})^2\,\frac{\left(\mu_{z}(t)^\frac{N-1}{N}\right)^2}{-\mu'_z(t)}\,dt\right)\,dz\\
&\ge \frac{C_1}{25}\,\mathcal{A}(\Omega)^2\,\int_0^{z_0} z^{1-2\,s}\,\left(\int_\frac{T}{4}^{\frac{3}{8}\,T}\frac{\left(\mu_{z}(t)^\frac{N-1}{N}\right)^2}{-\mu'_z(t)}\,dt\right)\,dz,
\end{split}
\]
where we used Proposition \ref{prop:asimmetrie} in the third inequality, which is possible thanks to \eqref{z_01}. 
\par
We observe that by using \eqref{eddai!} and the fact that $\mathcal{A}(\Omega)<2$, we get
\[
\mu_z(t)\ge 1-\frac{1}{3}\,\mathcal{A}(\Omega)>\frac{1}{3}, \qquad \mbox{ for every } \frac{T}{4}\le t\le \frac{3}{8}\,T,
\]
This in turn implies that
\[
\begin{split}
\lambda_{s,q}(\Omega)-\lambda_{s,q}(B)
&\ge \frac{C_1}{25}\,\left(\frac{1}{9}\right)^\frac{N-1}{N}\,\mathcal{A}(\Omega)^2\,\int_0^{z_0} z^{1-2\,s}\,\left(\int_\frac{T}{4}^{\frac{3}{8}\,T}\frac{1}{-\mu'_z(t)}\,dt\right)\,dz.
\end{split}
\]
In order to estimate the integral in $t$, we use Jensen's inequality
\[
\int_\frac{T}{4}^{\frac{3}{8}\,T}\frac{1}{-\mu'_z(t)}\,dt\ge \frac{T^2}{64}\,\frac{1}{\displaystyle\int_{\frac{T}{4}}^{\frac{3}{8}\,T} -\mu'_z(t)\,dt}\ge \frac{T^2}{64}\,\frac{1}{|E_{\frac{T}{4},z}|-|E_{\frac{3}{8}\,T,z}|}.
\]
By using \eqref{eddai!} with $t=T/4$ and $t=3\,T/8$, we get (recall that $|\Omega|=1$)
\[
|E_{\frac{T}{4},z}|-|E_{\frac{3}{8}\,T,z}|\le 1+\frac{1}{3}\,\mathcal{A}(\Omega)-\left(1-\frac{1}{3}\,\mathcal{A}(\Omega)\right)=\frac{2}{3}\,\mathcal{A}(\Omega).
\]
In conclusion, we obtain
\begin{equation}
\label{sgraffigna}
\begin{split}
\lambda_{s,q}(\Omega)-\lambda_{s,q}(B)&\ge \frac{3}{2}\,\frac{C_1}{25}\,\left(\frac{1}{9}\right)^\frac{N-1}{N}\,\mathcal{A}(\Omega)\,\frac{T^2}{64}\,\left(\int_0^{z_0} z^{1-2\,s}\, dz\right)\\
&=\frac{3}{2}\,\frac{C_1}{25}\,\left(\frac{1}{9}\right)^\frac{N-1}{N}\,\mathcal{A}(\Omega)\,\frac{T^2}{64}\,\frac{1}{2\,(1-s)}\,z_0^{2\,(1-s)}.
\end{split}
\end{equation}
By recalling the definition \eqref{z_0} of $z_0$ and that 
\[
T>T_0=\frac{C_2}{4\,(1+C_2)}\,\mathcal{A}(\Omega),
\] 
we get the desired conclusion in this case, as well. 

\begin{oss}
From the proof above, we can extract the following explicit value for $\sigma_1$
\[
\sigma_1=\min\left\{\frac{C_2}{1+C_2}\,\frac{(1-s)\,\lambda_{s,q}(B)}{2^\frac{3}{s}}, \frac{3}{256}\,\frac{C_1}{25}\,\left(\frac{1}{9}\right)^\frac{N-1}{N}\,\left(\frac{C_2}{4\,(1+C_2)}\right)^\frac{2}{s}\,\left(\frac{1}{2\cdot 576\,\beta_{N,s}\,\lambda_{s,q}(B)}\right)^\frac{1-s}{s}\right\},
\]
where $B$ is any ball with $|B|=1$. The constants $C_1=C_1(N,s)>0$ and $C_2=C_2(N,s,q)>0$ are given in \eqref{C1} and \eqref{C2}, respectively.
We then observe that: 
\begin{itemize}
\item by Remark \ref{oss:constants}
\[
\lim_{s\nearrow 1}\beta_{N,s}<+\infty,\qquad \mbox{ and }\qquad \lim_{s\nearrow 1} C_1=2\,N\,\omega_N^{1/N}\,\Theta_N\,\lim_{s\nearrow 1}\gamma_{N,s}>0;
\]
\item by definition
\[
\lim_{s\nearrow 1} C_2=\frac{1}{9}\,\left(\frac{2}{q}+\frac{2}{N}-1\right)>0;
\]
\item by Lemma \ref{lm:convergenza} below
\[
\lim_{s\nearrow 1} (1-s)\,\lambda_{s,q}(B)=\frac{\omega_N}{2}\,\lambda_{1,q}(B), 
\]
and 
\[
\lim_{s\nearrow 1} \Big(\lambda_{s,q}(B)\Big)^\frac{1-s}{s}=\lim_{s\nearrow 1} \exp\left(\frac{1-s}{s}\,\log\Big((1-s)\, \lambda_{s,q}(B)\Big)-\frac{1-s}{s}\,\log(1-s)\right)=1.
\]
\end{itemize}
This shows that $\sigma_1$ has the claimed stability property as $s\nearrow 1$.
\end{oss}

\subsection{Proof of Corollary \ref{coro:torsion}} We can suppose that $|\Omega|=1$. We then take $B$ a ball such that $|B|=|\Omega|=1$. Observe that if 
\[
\mathcal{T}_s(\Omega)\le \frac{1}{2}\,\mathcal{T}_s(B),
\] 
then
\[
\mathcal{T}_s(B)-\mathcal{T}_s(\Omega)\ge \frac{1}{2}\,\mathcal{T}_s(B)\ge \frac{\mathcal{T}_s(B)}{2^\frac{3+s}{s}}\,\mathcal{A}(\Omega)^\frac{3}{s}.
\]
As usual, we used that $\mathcal{A}(\Omega)<2$. This gives the desired stability estimate, under the standing assumption on $\mathcal{T}_s(\Omega)$. On the other hand, if
\begin{equation}
\label{torsionealta}
\mathcal{T}_s(\Omega)> \frac{1}{2}\,\mathcal{T}_s(B),
\end{equation}
we can use Theorem \ref{teo:main} with $q=1$
 \[
 \frac{1}{\mathcal{T}_s(\Omega)}-\frac{1}{\mathcal{T}_s(B)}\ge \frac{\sigma_1}{(1-s)}\,\mathcal{A}(\Omega)^\frac{3}{s},
 \]
 i.e.
 \[
 \frac{\mathcal{T}_s(B)-\mathcal{T}_s(\Omega)}{\mathcal{T}_s(B)\,\mathcal{T}_s(\Omega)}\ge \frac{\sigma_1}{(1-s)}\,\mathcal{A}(\Omega)^\frac{3}{s}.
 \]
By using \eqref{torsionealta}, we get
\[
\mathcal{T}_s(B)-\mathcal{T}_s(\Omega)\ge \frac{\sigma_1\,(\mathcal{T}_s(B))^2}{2\,(1-s)}\,\mathcal{A}(\Omega)^\frac{3}{s},
\]
which proves the stability estimate in this case, as well. The proof is complete.
\begin{oss}
An inspection of the proof shows that the constant $\sigma_2$ in Corollary \ref{coro:torsion} can be taken to be
\[
\sigma_2=\min\left\{\frac{1}{2^\frac{3+s}{s}}\,\frac{\mathcal{T}_s(B)}{1-s},\,\frac{\sigma_1}{2}\,\left(\frac{\mathcal{T}_s(B)}{1-s}\right)^2\right\}.
\]
where $B$ is a ball such that $|B|=1$. By observing that (see Lemma \ref{lm:convergenza})
\[
\lim_{s\nearrow 1}\frac{\mathcal{T}_s(B)}{1-s}=\frac{2}{\omega_N}\,\mathcal{T}_1(B),
\]
we get that the constant $\sigma_2$ has the claimed controlled behavior, as $s$ approaches $1$.
\end{oss}

\section{Smooth sets}
\label{sec:6}

In this section, we briefly explain how on smoother sets we can improve our quantitative estimate, by lowering the exponent on the Fraenkel asymmetry. We will use the same notation as before.
\par
We start by showing that when the trace has additional smoothness properties, we can upgrade the $L^2$ control of \eqref{lemma2} to an $L^\infty$ one.

\begin{lm}\label{C^s}
Let us suppose that $\varphi\in W^{s,2}(\mathbb{R}^N)\cap L^{(2^*_{1-s})'}(\mathbb{R}^N)$ is such that
\[
[\varphi]_{0,s}:=\sup_{x\in\mathbb{R}^N}\sup_{|h|>0} \left|\frac{\varphi(x+h)-\varphi(x)}{|h|^s}\right|<+\infty.
\]
Then we have
\[
\|U_\varphi(\cdot,z)-\varphi\|_{L^\infty(\mathbb{R}^N)}\le C\,[\varphi]_{0,s}\,z^s,
\]
for a constant $C=C(N,s)>0$.
\end{lm}
\begin{proof}
By using that the Poisson kernel has integral equal to $1$, we have
\[
\begin{split}
|U_\varphi(x,z)-\varphi(x)|&=\left|\int_{\mathbb{R}^{N}} P_z(y)\,[\varphi(x-y)-\varphi(x)]\,dy\right|\\
&\le\int_{\mathbb{R}^N} P_z(y)\,|\varphi(x-y)-\varphi(x)|\,dy\\
&\le [\varphi]_{0,s}\,\int_{\mathbb{R}^N}P_z(y)\,|y|^s\,dy=z^s\,[\varphi]_{0,s}\,\int_{\mathbb{R}^N} P_1(w)\,|w|^s\,dw.
\end{split}
\]
By defining
\[
C=\int_{\mathbb{R}^N} P_1(w)\,|w|^s\,dw,
\]
we get the desired conclusion.
\end{proof}

\begin{lm}[Closeness of level sets, $L^\infty$ case]\label{L-infty}
Let $\Omega\subset\mathbb{R}^N$ be an open bounded set and let $0<s<1$. Let us suppose that there exists a constant $C>0$ such that
\begin{equation}
\label{stima}
\|U_\Omega(\cdot,z)-u_\Omega\|_{L^\infty(\mathbb{R}^N)}\le C\,z^{s},\qquad \mbox{ for } z>0.
\end{equation}
We fix $T>0$ as in \eqref{livellone}, then  
\[
\mbox{ for every }\quad \frac{T}{4}\le t\le \frac{3}{8}\,T\quad \mbox{ and every }\quad   0<z\le \left(\frac{T}{8\,C}\right)^\frac{1}{s},
\]
we have
\[
\left\{x\in\Omega\, :\, u_\Omega(x)>\frac{T}{2}\right\}\subset E_{t,z}\subset \left\{x\in\Omega\, :\, u_\Omega(x)>\frac{T}{8}\right\}.
\]
In particular, for $t$ and $z$ as above, it holds $E_{t,z}\subset \Omega$ with
\[
|\Omega|-|E_{t,z}|\le \frac{1}{9}\,|\Omega|\,\mathcal{A}(\Omega),
\]
and
\[
\mathcal{A}\big(E_{t,z}\big)\ge \frac{7}{11}\,\mathcal{A}(\Omega).
\]
\end{lm}
\begin{proof}
We take a point $x\in \Omega$ such that $u_\Omega(x)>T/2$.
By using \eqref{stima}, we get for every $0<z<(T/(8\,C))^{1/s}$
\[
U_\Omega(x,z)\ge u_\Omega(x)-C\,z^s>\frac{T}{2}-C\,z^s\ge \frac{3}{8}\,T.
\]
This shows that for every $T/4\le t\le 3\,T/8$, we have
\[
\left\{x\in\Omega\, :\, u_\Omega(x)>T/2\right\}\subset E_{\frac{3}{8}\,T,z}\subset E_{t,z},
\]
where for the second inclusion we used the monotonicity of the level sets. This shows the validity of the first claimed inclusion.
\par
We now take $x\in\mathbb{R}^N$ such that $U_\Omega(x,z)>t$ 
and use again \eqref{stima}. We get for every $0<z<(T/(8\,C))^{1/s}$
\[
u_\Omega(x)\ge U_\Omega(x,z)-C\,z^s>\frac{T}{4}-C\,z^s\ge \frac{T}{8}.
\]
This shows that for every $T/4\le t\le 3\,T/8$
\[
E_{t,z}\subset \left\{x\in\Omega\, :\, u_\Omega(x)>\frac{T}{8}\right\},
\]
as desired.
\par
In order to prove the lower bound on the asymmetry of $E_{t,z}$, it is sufficient to reproduce the proof of \eqref{eddai!} and observe that this time
\[
|\Omega_{T/2}\setminus E_{t,z}|=|E_{t,z}\setminus \Omega_{T/8}|=0.
\]
This gives
\[
\frac{|E_{t,z}\Delta \Omega|}{|\Omega|}\le \frac{|\Omega\setminus \Omega_{T/2}|}{|\Omega|}\le \frac{1}{9}\,\mathcal{A}(\Omega),
\]
thanks to the choice \eqref{livellone} of $T$. We now get the conclusion by applying Lemma \ref{lm:asimmetria} with $\gamma=1/9$.
\end{proof}

Then for regular sets $\Omega\subset\mathbb{R}^N$ we can slightly improve the exponent on the asymmetry in our quantitative estimate, according to the following

\begin{teo}\label{teo:smooth}
%Let $N\ge 2$, $0<s<1$ and $1\le q<2^*_s$. Let $\Omega\subset\mathbb{R}^N$ be an open bounded set, having Lipschitz boundary and satisfying the exterior ball condition (with radius $\rho$), or having $C^{1,\alpha}$ boundary for $\alpha>0$. Then we have
Let $N\ge 2$, $0<s<1$ and $1\le q<2^*_s$. Let $\Omega\subset\mathbb{R}^N$ be an open bounded set, satisfying one of the following conditions:
\begin{itemize}
\item[A.] either $\partial\Omega$ is Lipschitz and $\Omega$ satisfies the exterior ball condition, with radius $\rho$;
\vskip.2cm
\item[B.] or $\partial \Omega$ is $C^{1,\alpha}$, for some $0<\alpha<1$.
\end{itemize} 
Then we have
\[
|\Omega|^{\frac{2}{q}-1+\frac{2\,s}{N}}\,\lambda_{s,q}(\Omega)-|B|^{\frac{2}{q}-1+\frac{2\,s}{N}}\,\lambda_{s,q}(B)\ge \frac{C}{1-s}\,\mathcal{A}(\Omega)^{2+\frac{1}{s}},
\]
for a constant $C>0$ depending on $N,\,s,\,q$ and $\rho$ and the Lipschitz constant of $\partial \Omega$ (case {\rm A}) or the $C^{1,\alpha}$ norm of $\partial\Omega$ (case {\rm B}).
\end{teo}

\begin{proof}

We start by observing that, under the standing assumptions, $u_\Omega$ is a function of class $C^s(\mathbb R^N)$, thanks to \cite[Proposition 1.1]{RS} (when assumption A is in force) or by \cite[Proposition 1.1]{RS2} (if assumption B is taken). Moreover, the $C^s-$norm of $u_\Omega$ is bounded by a constant that depends on $N,\,s,\,q$, on the relevant regularity parameters of $\partial \Omega$ and on $\lambda_{s,q}(\Omega)$.
Hence, by Lemma \ref{C^s}, we have that assumption \eqref{stima} of Lemma \ref{L-infty} is satisfied with a constant $C$ which depends on the above quantities. As before, it is sufficient to perform the proof under the restriction \eqref{restrizione}, thus we observe that the dependence on $\lambda_{s,q}(\Omega)$ can be removed.
\par
We proceed now as in the proof of Theorem \ref{teo:main}, by using the same notations. The case $T \le T_0$ follows exactly as before.

For the case $T>T_0$, we use Lemma \ref{L-infty} (in place of Lemma \ref{lm:stimaL1}) and we set 
\[
z_1=\left(\frac{T}{8\,C}\right)^{\frac{1}{s}}.
\]
By proceeding as before, we now get
 \[
\begin{split}
\lambda_{s,q}(\Omega)-\lambda_{s,q}(B)&\ge C_1\,\int_0^{z_1} z^{1-2\,s}\,\left(\int_\frac{T}{4}^{\frac{3}{8}\,T} \mathcal{A}(E_{t,z})^2\,\frac{\left(\mu_{z}(t)^\frac{N-1}{N}\right)^2}{-\mu'_z(t)}\,dt\right)\,dz\\
&\ge \frac{49\,C_1}{121}\,\mathcal{A}(\Omega)^2\,\int_0^{z_1} z^{1-2\,s}\,\left(\int_\frac{T}{4}^{\frac{3}{8}\,T}\frac{\left(\mu_{z}(t)^\frac{N-1}{N}\right)^2}{-\mu'_z(t)}\,dt\right)\,dz.
\end{split}
\]
By arguing as in the proof of Theorem \ref{teo:main}, we deduce that 
\[
\begin{split}
\lambda_{s,q}(\Omega)-\lambda_{s,q}(B)&\ge C\,\mathcal{A}(\Omega)\,T^2\,\left(\int_0^{z_1} z^{1-2\,s}\, dz\right)\\
&\ge \frac{C}{2\,(1-s)}\,\mathcal{A}(\Omega)\,T^2_0\,z_1^{2\,(1-s)}. 
\end{split}
\]
By recalling the definitions of $z_1$ and $T_0$, we get the conclusion.  
\end{proof}
\begin{oss}
Observe that, differently from the proof of Theorem \ref{teo:main}, the level $z_1$ does not depend on the asymmetry itself. This explains why the resulting exponent on $\mathcal{A}(\Omega)$ is smaller. Also observe that even this improved exponent converges to $3$, as $s$ goes to $1$.
\end{oss}
\begin{oss}[Fractional torsional rigidity]
By taking $q=1$ in Theorem \ref{teo:smooth} and recalling that
\[
\mathcal{T}_s(\Omega)=\frac{1}{\lambda_{s,1}(\Omega)},
\]
we can obtain
\begin{equation}
\label{kim}
\frac{\mathcal{T}_s(B)}{|B|^\frac{N+2\,s}{N}}-\frac{\mathcal{T}_s(\Omega)}{|\Omega|^\frac{N+2\,s}{N}}\ge C'\,\mathcal{A}(\Omega)^{2+\frac{1}{s}},
\end{equation} 
for every open bounded set $\Omega\subset\mathbb{R}^N$ which satisfies the assumptions of Theorem \ref{teo:smooth}. %with Lipschitz boundary and which satisfies the exterior ball condition or with $C^{1,\alpha}$ boundary.
It is sufficient to repeat the proof of Corollary \ref{coro:torsion} and use Theorem \ref{teo:smooth}, in place of Theorem \ref{teo:main}.
\par
We point out that, after the completion of this paper, the manuscript \cite{Ki} appeared. There the author proved the very same estimate \eqref{kim}, without any further regularity assumption on $\Omega$, see \cite[Theorem 1.8]{Ki}. The proof in \cite{Ki} starts from our estimate \eqref{sgraffigna} for $q=1$ and exploits the possibility of writing $\mathcal{T}_s(\Omega)$ as an integral of the distribution function $t\mapsto|\{x\, :\, u_\Omega>t\}|$, in order to get a (slightly) better control in terms of $\mathcal{A}(\Omega)$. However, this approach can not be generalized to cover all the range $1<q<2^*_s$.
\end{oss}

\appendix

\section{Asymptotics for the Poincar\'e-Sobolev constant}

In the next result, we use $2^*$ to denote the usual Sobolev exponent, i.e.
\[
2^*=\left\{\begin{array}{cc}
\dfrac{2\,N}{N-2},&\mbox{ if }N\ge 3,\\
&\\
+\infty,& \mbox{ if }N=2.
\end{array}
\right.
\]
For $1\le q<2^*$, we recall the notation
\[
\lambda_{1,q}(\Omega)=\min_{u\in \mathcal{D}^{1,2}_0(\Omega)} \Big\{\|\nabla u\|_{L^2(\Omega)}^2\, :\, \|u\|_{L^q(\Omega)}=1\Big\}.
\]
\begin{lm}
\label{lm:convergenza}
Let $N\ge 2$ and $1\le q<2^*$, then for every $\Omega\subset\mathbb{R}^N$ open bounded set, we have 
\begin{equation}
\label{limsup}
\limsup_{s\nearrow 1} (1-s)\,\lambda_{s,q}(\Omega)\le \frac{\omega_N}{2}\, \lambda_{1,q}(\Omega).
\end{equation}
If in addition $\Omega$ has Lipschitz boundary, then
\begin{equation}
\label{lim}
\lim_{s\nearrow 1} (1-s)\,\lambda_{s,q}(\Omega)=\frac{\omega_N}{2}\,\lambda_{1,q}(\Omega).
\end{equation}
\end{lm}
\begin{proof}
In order to prove \eqref{limsup}, it is sufficient to use the Bourgain-Brezis-Mironescu convergence result
\[
\lim_{s\nearrow 1}(1-s)\,[\varphi]^2_{W^{s,2}(\mathbb{R}^N)}=\frac{\omega_N}{2}\,\int_\Omega |\nabla \varphi|^2\,dx,\qquad \mbox{ for every } \varphi\in C^\infty_0(\Omega),
\]
see \cite{bourgain}. 
Indeed, by taking $\varphi\in C^\infty_0(\Omega)$ with unit $L^q$ norm and using the definition of $\lambda_{s,q}(\Omega)$, from the previous formula we get
\[
\limsup_{s\nearrow 1}(1-s)\,\lambda_{s,q}(\Omega)\le \lim_{s\nearrow 1}(1-s)\,[\varphi]^2_{W^{s,2}(\mathbb{R}^N)}=\frac{\omega_N}{2}\,\int_\Omega |\nabla \varphi|^2\,dx.
\]
By taking the infimum over all admissible $\varphi$, we get \eqref{limsup}.
\vskip.2cm\noindent 
We now show \eqref{lim}. The case $q=2$ is already contained in \cite[Theorem 1.2]{BPS}, we thus treat the case $q\not =2$.  We take $1\le q<2^*$ and fix
\[
\delta=\frac{1}{q}-\left(\frac{1}{2}-\frac{1}{N}\right)>0.
\] 
We then observe that
\[
s_0:=1-N\,\delta<s<1\qquad \Longrightarrow\qquad q<2^*_s.
\]
From now on, we thus work with $s_0<s<1$.
 We start by observing that \cite[Corollary 2.2]{BPS} entails
\[
s\,(1-s)\,[u]^2_{W^{s,2}(\mathbb{R}^N)}\le  C\,\|u\|_{L^2(\Omega)}^{2\,(1-s)}\,\|\nabla u\|_{L^2(\Omega)}^{2\,s},\qquad \mbox{ for every } u\in C^\infty_0(\Omega),
\]
for some $C=C(N)>0$.
In particular, by using the definition of $\lambda_{s,q}(\Omega)$, we get
\begin{equation}
\label{eccoci}
s\,(1-s)\,\lambda_{s,q}(\Omega)\,\|u\|_{L^q(\Omega)}^2\le C\,\|u\|_{L^2(\Omega)}^{2\,(1-s)}\,\|\nabla u\|_{L^2(\Omega)}^{2\,s}.
\end{equation}
We have to distinguish two cases: 
\begin{itemize}
\item if $\boxed{1\le q<2}$, we use in \eqref{eccoci} the Gagliardo-Nirenberg inequality
\[
\|u\|_{L^2(\Omega)}\le C_{N,q}\,\|u\|_{L^q(\Omega)}^{1-\vartheta}\,\|\nabla u\|_{L^2(\Omega)}^{\vartheta},
\]
where $\vartheta$ is determined by scale invariance, i.e.
\[
\vartheta=\frac{(2-q)\,N}{2\,N-q\,(N-2)}.
\]
This yields
\[
s\,(1-s)\,\lambda_{s,q}(\Omega)\,\|u\|_{L^q(\Omega)}^2\le C\,\|u\|_{L^q(\Omega)}^{2\,(1-s)\,(1-\vartheta)}\,\|\nabla u\|_{L^2(\Omega)}^{2\,(1-s)\,\vartheta+2\,s},
\]
for a possibly different constant $C=C(N,q)>0$.
By dividing on both sides by $\|u\|^2_{L^q}$, we get
\[
s\,(1-s)\,\lambda_{s,q}(\Omega)\le C\, \left(\frac{\|\nabla u\|^2_{L^2(\Omega)}}{\|u\|_{L^q(\Omega)}^2}\right)^{(1-s)\,\vartheta+s}.
\]
Since this holds for every $u\in C^\infty_0(\Omega)$, we finally get
\begin{equation}
\label{eccoci1}
s\,(1-s)\,\lambda_{s,q}(\Omega)\le C\,\Big(\lambda_{1,q}(\Omega)\Big)^{(1-s)\,\vartheta+s}.
\end{equation}
\item if $\boxed{2<q<2^*}$, we use in \eqref{eccoci} H\"older's inequality
\[
\|u\|^2_{L^2(\Omega)}\le |\Omega|^{1-\frac{2}{q}}\,\|u\|^2_{L^q(\Omega)}.
\]
This now gives
\[
s\,(1-s)\,\lambda_{s,q}(\Omega)\,\|u\|_{L^q(\Omega)}^2\le C\, |\Omega|^{\left(1-\frac{2}{q}\right)\,(1-s)}\,\|u\|^{2\,(1-s)}_{L^q(\Omega)}\,\|\nabla u\|_{L^2(\Omega)}^{2\,s}.
\]
By proceeding as in the previous case, we thus obtain
\begin{equation}
\label{eccoci2}
s\,(1-s)\,\lambda_{s,q}(\Omega)\le C\, |\Omega|^\frac{(q-2)\,(1-s)}{q}\,\Big(\lambda_{1,q}(\Omega)\Big)^s.
\end{equation}
\end{itemize}
From \eqref{eccoci1} and \eqref{eccoci2}, we have obtained that there exists a constant $\widetilde C=\widetilde C(N,q,\Omega,\delta)>0$ such that
\begin{equation}
\label{uniforme}
s\,(1-s)\,\lambda_{s,q}(\Omega)\le \widetilde C,\qquad \mbox{ for every } s_0<s<1.
\end{equation}
Then for every sequence $\{s_k\}_{k\in\mathbb{N}}\subset (s_0,1)$ converging to $1$, we take 
\[
u_k\in\mathcal{D}^{s_k,2}_0(\Omega),\qquad u_k>0 \mbox{ in } \Omega,
\]
to be a minimizer of the variational problem which defines $\lambda_{s_k,q}(\Omega)$. By definition and estimate \eqref{uniforme}, we have 
\begin{equation}\label{boundk}
(1-s_k)\,[u_k]^2_{W^{s_k,2}(\mathbb{R}^N)}=(1-s_k)\,\lambda_{s_k,q}(\Omega)\le \widetilde C,\qquad \mbox{ for every } k\in\mathbb{N}.
\end{equation}
By using \cite[Lemma 3.10]{BPS}, up to consider a subsequence, we have
\begin{equation}
\label{strongL2}
\lim_{k\to\infty} \|u_{k}-u\|_{L^2(\Omega)}=0,
\end{equation}
for some $u\in \mathcal{D}^{1,2}_0(\Omega)$. With a simple argument, from the previous estimate we can also infer 
\[
\lim_{k\to\infty} \|u_k-u\|_{L^q(\Omega)}=0.
\]
Indeed, for $1\le q<2$ this simply follows from H\"older's inequality. For $q>2$, we can use the fractional Sobolev inequality and the interpolation inequality of \cite[Proposition 2.1]{BPS}, so to get
\begin{eqnarray*}
\|u_k-u\|^2_{L^q(\Omega)}&\le& \frac{1}{\lambda_{\beta,2}(\Omega)}\,[u_k-u]^2_{W^{\beta,2}(\mathbb{R}^N)}\\
&\le&C\,\frac{s_k}{s_k-\beta}\|u_k-u\|^{2\,(1-\frac{\beta}{s_k})}_{L^2(\Omega)}\,\left((1-s_k)\,[u_k-u]^2_{W^{s_k,2}(\mathbb{R}^N)}\right)^{\frac{\beta}{s_k}},
\end{eqnarray*}
where $0<\beta<s_k$ is a fixed exponent, taken so that $2^*_\beta=q$. Hence, by the fact that $u\in\mathcal{D}^{1,2}_0(\Omega)$ and using \eqref{boundk} and \eqref{strongL2}, we obtain also in this case that
\[
\lim_{k\to\infty} \|u_k-u\|_{L^q(\Omega)}\le \lim_{k\to\infty} C\,\|u_k-u\|^{1-\frac{\beta}{s_k}}_{L^2(\Omega)}=0,
\]
since the constant $C>0$ in the last inequality can be taken uniform as $k$ goes to $\infty$.
\par
In particular, the function $u$ has unit $L^q$ norm. By using the $\Gamma-$convergence result of \cite[Proposition 3.11]{BPS} and the minimality of $u_{k}$, we get
\begin{equation}
\label{quasiquasi}
\begin{split}
\frac{\omega_N}{2}\,\lambda_{1,q}(\Omega)\le \frac{\omega_N}{2}\,\int_\Omega |\nabla u|^2\,dx&\le \liminf_{k\to\infty} (1-s_k)\,[u_{k}]^2_{W^{s_k,2}(\mathbb{R}^N)}\\
&=  \liminf_{k\to\infty} (1-s_k)\,\lambda_{s_k,q}(\Omega).
\end{split}
\end{equation}
By using \eqref{limsup}, we thus get
\[
\frac{\omega_N}{2}\,\lambda_{1,q}(\Omega)=\lim_{k\to\infty} (1-s_k)\,\lambda_{s_k,q}(\Omega).
\]
This in turn implies that equality must hold everywhere in \eqref{quasiquasi}, thus the limit function $u$ is optimal for $\lambda_{1,q}(\Omega)$. This concludes the proof.
\end{proof}
\begin{oss}[Irregular sets]
The hypothesis of Lipschitz regularity on $\partial\Omega$ could probably be relaxed. However, for a general open bounded set $\Omega\subset\mathbb{R}^N$ the equality \eqref{lim} is not true. Indeed, we can produce a counter-example by using a similar construction to that of \cite[Section 7]{Lqv}, which deals with a related phenomenon. 
\par
By using \cite[Section 10.4.3, Proposition 5]{maz}, we can exhibit a Cantor--type bounded set $F\subset\mathbb{R}^N$ such that
\[
\mathrm{cap}_1(F)>0\qquad \mbox{ but }\qquad \mathrm{cap}_s(F)=0 \mbox{ for every } s<1.
\]
Here $\mathrm{cap}_s(F)$ denotes the $s-$capacity of $F$. In this way, if we consider the open set $B\setminus F$, the set $F$ is ``invisible'' for every $\lambda_{s,q}(B\setminus F)$, when $s<1$. Then we get 
\[
\lim_{s\nearrow 1} (1-s)\,\lambda_{s,q}(B\setminus F)=\lim_{s\nearrow 1} (1-s)\,\lambda_{s,q}(B)=\frac{\omega_N}{2}\,\lambda_{1,q}(B)<\frac{\omega_N}{2}\,\lambda_{1,q}(B\setminus F).
\]
The last strict inequality follows from the fact that $F$ has positive capacity when $s=1$.
\end{oss}

\begin{lm}\label{lm:towardsS}
Let $N\ge 2$ and $0<s<1$. We define the sharp Sobolev constant in $\mathbb{R}^N$
\[
\mathcal{S}_{N,s}:=\lambda_{s,2^*_s}(\mathbb{R}^N)=\inf_{u\in C^\infty_0(\mathbb{R}^N)} \left\{[u]^2_{W^{s,2}(\mathbb{R}^N)}\, :\, \int_{\mathbb{R}^N} |u|^{2^*_s}\,dx=1\right\}.
\] 
Then for every $\Omega\subset\mathbb{R}^N$ open set with finite measure, we have
\[
\lim_{q\nearrow 2^*_s} \lambda_{s,q}(\Omega)=\mathcal{S}_{N,s}.
\]
\end{lm}
\begin{proof}
The proof is standard, we give it for completeness.
We set
\[
\mathcal{U}(\varrho)=\left(1+\varrho\right)^\frac{2\,s-N}{2},
\]
then we know that functions of the form
\[
U_{t,x_0}(x)=\mathcal{U}\left(\frac{|x-x_0|}{t}\right),\qquad t>0,\,x_0\in\mathbb{R}^N,
\] 
are such that
\begin{equation}
\label{estremali}
\frac{[U_{t,x_0}]^2_{W^{s,2}(\mathbb{R}^N)}}{\displaystyle\left(\int_{\mathbb{R}^N} |U_{t,x_0}|^{2^*_s}\,dx\right)^\frac{2}{2^*_s}}=\mathcal{S}_{N,s},
\end{equation}
see \cite[Theorem 1.1]{CT}.
Since $\Omega$ is an open set, there exists $B_R(x_0)\subset\Omega$. We consider the ``truncated'' extremal
\[
\varphi_t(x)=\left(\mathcal{U}\left(\frac{|x-x_0|}{t}\right)-\mathcal{U}\left(\frac{R}{t}\right)\right)_+,\qquad x\in\Omega.
\]
Then the function $\varphi_t/\|\varphi_t\|_{L^q(\Omega)}$ is admissible in the variational problem which defines $\lambda_{s,q}(\Omega)$. Thus we get
\[
\begin{split}
\limsup_{q\nearrow 2^*_s}\lambda_{s,q}(\Omega)\le \limsup_{q\nearrow 2^*_s}\frac{[\varphi_t]_{W^{s,2}(\mathbb{R}^N)}^2}{\|\varphi_t\|_{L^q(\Omega)}^2}&\le \frac{[U_{t,x_0}]_{W^{s,2}(\mathbb{R}^N)}^2}{\left(\displaystyle\int_{B_R(x_0)} \left(\mathcal{U}\left(\frac{|x-x_0|}{t}\right)-\mathcal{U}\left(\frac{R}{t}\right)\right)^{2^*_s}\,dx\right)^\frac{2}{2^*_s}}\\
&=\frac{t^{N-2\,s}\,[\mathcal{U}]_{W^{s,2}(\mathbb{R}^N)}^2}{t^{N-2\,s}\left(\displaystyle\int_{B_\frac{R}{t}(0)} \left(\mathcal{U}(|x|)-\mathcal{U}\left(\frac{R}{t}\right)\right)^{2^*_s}\,dx\right)^\frac{2}{2^*_s}}.
\end{split}
\]
By taking the limit as $t$ goes to $0$ and recalling \eqref{estremali}, we can infer
\[
\limsup_{q\nearrow 2^*_s}\lambda_{s,q}(\Omega)\le \mathcal{S}_{N,s}.
\]
In order to prove that
\[
\liminf_{q\nearrow 2^*_s}\lambda_{s,q}(\Omega)\ge \mathcal{S}_{N,s},
\]
it is sufficient to use H\"older's inequality. Indeed, this gives that
\[
\lambda_{s,2^*_s}(\Omega)\le |\Omega|^{\frac{2}{q}-\frac{2}{2^*_s}}\,\lambda_{s,q}(\Omega),\qquad \mbox{ for } q<2^*_s.
\]
By observing that $\lambda_{s,2^*_s}(\mathbb{R}^N)\le \lambda_{s,2^*_s}(\Omega)$, we can now get the desired conclusion.
\end{proof}

\section{Two equivalent seminorms in $W^{s,2}(\mathbb{R}^N)$}
\label{sec:B}

For every measurable function $u:\mathbb{R}^N\to\mathbb{R}$ and every $i=1,\dots,N$, we define
\[
[u]_{w_i^{s,2}(\mathbb{R}^N)}=\left(\int_{0}^{+\infty} \int_{\mathbb{R}^N} \frac{|u(x+\varrho\,\mathbf{e}_i)-u(x)|^2}{\varrho^{1+2\,s}}\,dx\,d\varrho\right)^\frac{1}{2}.
\]
In Lemma \ref{prop:space} we used that the sum of these seminorms is equivalent to the standard Sobolev-Slobodecki\u{\i} seminorm. Even if this result should belong to the folklore on Sobolev spaces, we have not been able to find a reference for this fact. 
\begin{prop}
\label{prop:equiRN}
Let $0<s<1$, then for every $u\in C^\infty_0(\mathbb{R}^N)$ we have
\[
\frac{1}{C}\,\sum_{i=1}^N [u]_{w_i^{s,2}(\mathbb{R}^N)}\le [u]_{W^{s,2}(\mathbb{R}^N)}\le C\,\sum_{i=1}^N [u]_{w_i^{s,2}(\mathbb{R}^N)},
\]
for a constant $C=C(N)>1$.
\end{prop}
\begin{proof}
We prove the two inequalities separately. In order to prove the first one, we define the $K-$functionals
\[
K(t,u)=\inf_{v\in C^\infty_0(\mathbb{R}^N)}\Big[ \|u-v\|_{L^2(\mathbb{R}^N)} +t\,\|\nabla v\|_{L^2(\mathbb{R}^N)}\Big],
\]
and
\[
K_i(t,u)=\inf_{v\in C^\infty_0(\mathbb{R}^N)}\Big[ \|u-v\|_{L^2(\mathbb{R}^N)} +t\,\|v_{x_i}\|_{L^2(\mathbb{R}^N)}\Big],\qquad i=1,\dots,N.
\]
Observe that we trivially have
\[
K_i(t,u)\le K(t,u),\qquad i=1,\dots,N.
\]
By using this and \cite[Theorem 35.2]{Ta}, we have 
\[
\sum_{i=1}^N\left(\int_0^\infty \left(\frac{K_i(t,u)}{t^s}\right)^2\,\frac{dt}{t}\right)^\frac{1}{2}\le C\,[u]_{W^{s,2}(\mathbb{R}^N)},
\]
thus in order to prove the first inequality in the statement, we only need to prove that
\begin{equation}
\label{low}
[u]_{w_i^{s,2}(\mathbb{R}^N)}\le C\,\left(\int_0^\infty \left(\frac{K_i(t,u)}{t^s}\right)^2\,\frac{dt}{t}\right)^\frac{1}{2},\qquad \mbox{ for } i=1,\dots,N.
\end{equation}
To prove \eqref{low}, we take $\varepsilon>0$ and $\varrho>0$, then there exists $v\in C^\infty_0(\mathbb{R}^N)$ such that
\begin{equation}
\label{dai}
\|u-v\|_{L^2(\mathbb{R}^N)}+\frac{\varrho}{2}\,\|v_{x_i}\|_{L^2(\mathbb{R}^N)}\le (1+\varepsilon)\,K_i\left(\frac{\varrho}{2},u\right).
\end{equation}
Thus we get\footnote{In the second inequality, we use the classical fact
\[
\left(\int_{\mathbb{R}^N} \frac{|v(x+\varrho\,\mathbf{e}_i)-v(x)|^2}{|\varrho|^2}\,dx\right)^\frac{1}{2}\le \left(\int_{\mathbb{R}^N} |v_{x_i}|^2\,dx\right)^\frac{1}{2}.
\]}
\[
\begin{split}
\left(\int_{\mathbb{R}^N} \frac{|u(x+\varrho\,\mathbf{e}_i)-u(x)|^2}{\varrho^{1+2\,s}}\,dx\right)^\frac{1}{2}&\le \left(\int_{\mathbb{R}^N} \frac{|u(x+\varrho\,\mathbf{e}_i)-v(x+\varrho\,\mathbf{e}_i)-u(x)+v(x)|^2}{\varrho^{1+2\,s}}\,dx\right)^\frac{1}{2}\\
&+\left(\int_{\mathbb{R}^N} \frac{|v(x+\varrho\,\mathbf{e}_i)-v(x)|^2}{\varrho^{1+2\,s}}\,dx\right)^\frac{1}{2}\\
&\le 2\,\varrho^{-\frac{1}{2}-s}\,\|u-v\|_{L^2(\mathbb{R}^N)}+\varrho^{1-\frac{1}{2}-s}\,\|v_{x_i}\|_{L^2(\mathbb{R}^N)}\\
&= 2\,\varrho^{-\frac{1}{2}-s}\,\left(\|u-v\|_{L^2(\mathbb{R}^N)}+\frac{\varrho}{2}\,\|v_{x_i}\|_{L^2(\mathbb{R}^N)}\right).
\end{split}
\]
By using \eqref{dai}, we then obtain
\[
\int_{\mathbb{R}^N} \frac{|u(x+\varrho\,\mathbf{e}_i)-u(x)|^2}{\varrho^{1+2\,s}}\,dx\le 4\,(1+\varepsilon)^2\,\varrho^{-1-2\,s}\,K_i\left(\frac{\varrho}{2},u\right)^2.
\]
We now integrate with respect to $\varrho>0$ and make a change of variable. This yields directly \eqref{low}, as desired.
\vskip.2cm\noindent
In order to prove the second inequality, we observe that
\[
\begin{split}
[u]^2_{W^{s,2}(\mathbb{R}^N)}&=\int_{\mathbb{R}^N}\left( \int_{\mathbb{R}^N} \frac{|u(x+h)-u(x)|^2}{|h|^{2\,s}}\,dx\right)\,\frac{dh}{|h|^N}\\
&=\int_{\mathbb{S}^{N-1}} \int_{0}^{+\infty} \left( \int_{\mathbb{R}^N} \frac{|u(x+\varrho\,\omega)-u(x)|^2}{\varrho^{2\,s}}\,dx\right)\,\frac{d\varrho}{\varrho}\,d\mathcal{H}^{N-1}.
\end{split}
\]
For $\omega=(\omega_1,\dots,\omega_N)\in\mathbb{S}^{N-1}$, we use the triangle inequality so to get
\[
\begin{split}
\int_0^{+\infty} \int_{\mathbb{R}^N}& \frac{|u(x+\varrho\,\omega)-u(x)|^2}{\varrho^{1+2\,s}}\,dx\,d\varrho\\
&=\int_0^{+\infty} \int_{\mathbb{R}^N} \frac{\left|u\left(x+\sum_{i=1}^N\varrho\,\omega_i\,\mathbf{e}_i\right)-u(x)\right|^2}{\varrho^{1+2\,s}}\,dx\,d\varrho\\
&\le N\, \sum_{j=0}^{N-1} \int_0^{+\infty} \int_{\mathbb{R}^N} \frac{\left|u\left(x+\sum_{i=1}^{N-j}\varrho\,\omega_i\,\mathbf{e}_i\right)-u\left(x+\sum_{i=1}^{N-j-1}\varrho\,\omega_i\,\mathbf{e}_i\right)\right|^2}{\varrho^{1+2\,s}}\,dx\,d\varrho \\
&=N\, \sum_{j=1}^{N} \int_{0}^{+\infty} \int_{\mathbb{R}^N} \frac{\left|u\left(x+\varrho\,\omega_{j}\,\mathbf{e}_{j}\right)-u(x)\right|^2}{\varrho^{1+2\,s}}\,dx\,d\varrho, 
\end{split}
\]
where we used the simple change of variable
\[
x+\sum_{i=1}^{N-j-1}\varrho\,\omega_i\,\mathbf{e}_i=y.
\]
We thus obtain
\[
\begin{split}
[u]^2_{W^{s,2}(\mathbb{R}^N)}&\le N\,\sum_{j=1}^{N} \int_{\mathbb{S}^{N-1}}\int_{0}^{+\infty} \int_{\mathbb{R}^N} \frac{\left|u\left(x+\varrho\,\omega_{j}\,\mathbf{e}_{j}\right)-u(x)\right|^2}{\varrho^{1+2\,s}}\,dx\,d\varrho\\
&=N\,\sum_{j=1}^{N} \int_{\{\omega_j>0\}\cap \mathbb{S}^{N-1}}\int_{0}^{+\infty} \int_{\mathbb{R}^N} \frac{\left|u\left(x+\varrho\,\omega_{j}\,\mathbf{e}_{j}\right)-u(x)\right|^2}{\varrho^{1+2\,s}}\,dx\,d\varrho\\
&+N\,\sum_{j=1}^{N} \int_{\{\omega_j<0\}\cap \mathbb{S}^{N-1}}\int_{0}^{+\infty} \int_{\mathbb{R}^N} \frac{\left|u\left(x+\varrho\,\omega_{j}\,\mathbf{e}_{j}\right)-u(x)\right|^2}{\varrho^{1+2\,s}}\,dx\,d\varrho.
\end{split}
\]
We can use the change of variable $\omega_j\,\varrho=t$ in the first integral and $\omega_j\,\varrho=-t$ in the second one,
so to obtain
\[
\begin{split}
[u]^2_{W^{s,2}(\mathbb{R}^N)}&\le N\, \sum_{j=1}^{N} \left(\int_{\{\omega_j>0\}\cap \mathbb{S}^{N-1}}\omega_j^{2\,s}\,d\mathcal{H}^{N-1}\right)\,\int_{0}^{+\infty} \int_{\mathbb{R}^N} \frac{\left|u\left(x+t\,\mathbf{e}_{j}\right)-u(x)\right|^2}{t^{1+2\,s}}\,dx\,dt\\
&+N\, \sum_{j=1}^{N} \left(\int_{\{\omega_j>0\}\cap \mathbb{S}^{N-1}}\omega_j^{2\,s}\,d\mathcal{H}^{N-1}\right)\,\int_{0}^{+\infty} \int_{\mathbb{R}^N} \frac{\left|u\left(x-t\,\mathbf{e}_{j}\right)-u(x)\right|^2}{t^{1+2\,s}}\,dx\,dt.
\end{split}
\]
In conclusion, we obtained
\[
[u]_{W^{s,2}(\mathbb{R}^N)}\le \sqrt{2\,N}\,\left(\int_{\{\omega_1>0\}\cap \mathbb{S}^{N-1}}\,|\omega_1|^{2\,s}\,d\mathcal{H}^{N-1}\right)^\frac{1}{2}\,\left(\sum_{j=1}^N [u]^2_{w_j^{s,2}(\mathbb{R}^N)}\right)^\frac{1}{2}.
\]
By using some standard algebraic manipulations, we then obtain the desired inequality.
\end{proof}

\end{document}